\documentclass[11pt]{article}
\usepackage{graphicx}
\usepackage{amssymb, amsthm, amsmath, amsfonts}
\usepackage{mathrsfs,microtype}
\usepackage{epstopdf,color}
\usepackage{tikz}
\usepackage{hyperref}
\usepackage{todonotes}
\usepackage{enumitem}

\usetikzlibrary{arrows,decorations.pathmorphing,decorations.footprints,fadings,calc,trees,mindmap,shadows,decorations.text,patterns,positioning,shapes,matrix,fit,through,decorations.pathreplacing}
\usepackage[margin=1in]{geometry}

\theoremstyle{plain}
\newtheorem{theorem}{Theorem}[section]
\newtheorem{lemma}[theorem]{Lemma}
\newtheorem{corollary}[theorem]{Corollary}

\newtheorem{conjecture}[theorem]{Conjecture}

\theoremstyle{definition}
\newtheorem{definition}[theorem]{Definition}

\newtheorem*{notation}{Notation}

\theoremstyle{remark}

\title{Tomescu's graph coloring conjecture for $\ell$-connected graphs}
\author{ John Engbers\thanks{Department of Mathematical and Statistical Sciences, Marquette University, Milwaukee WI, 53201, USA. Email: \texttt{john.engbers@marquette.edu}. Research supported by the Simons Foundation grant 524418.}
\and Aysel Erey\thanks{Department of Mathematics, Gebze Technical University, Kocaeli, Turkey. Email: \texttt{aysel.erey@gtu.edu.tr}. Research supported by T{\"U}B{\.I}TAK grant 118C009.} \and Jacob Fox\thanks{Department of Mathematics, Stanford University, Stanford,
CA 94305, USA. Email: {\tt jacobfox@stanford.edu}. Research supported by
a Packard Fellowship and by NSF grant DMS-1855635.} \and Xiaoyu He\thanks{Department of Mathematics, Stanford University, Stanford,
CA 94305, USA. Email: {\tt alkjash@stanford.edu}. Research supported by a NSF GRFP grant number DGE-1656518.}}
\date{\today }

\begin{document}

\maketitle

\begin{abstract}
Let $P_G(k)$ be the number of proper $k$-colorings of a finite simple graph $G$. Tomescu's conjecture, which was recently solved by Fox, He, and Manners, states that $P_G(k) \le k!(k-1)^{n-k}$ for all connected graphs $G$ on $n$ vertices with chromatic number $k\geq 4$. In this paper, we study the same problem with the additional constraint that $G$ is $\ell$-connected. For $2$-connected graphs $G$, we prove a tight bound
\[
P_G(k) \le (k-1)!((k-1)^{n-k+1} + (-1)^{n-k}),
\]
and show that equality is only achieved if $G$ is a $k$-clique with an ear attached. For $\ell \ge 3$, we prove an asymptotically tight upper bound
\[
P_G(k) \le k!(k-1)^{n-\ell - k + 1} + O((k-2)^n), 
\]
and provide a matching lower bound construction. For the ranges $k \geq \ell$ or $\ell \geq (k-2)(k-1)+1$ we further find the unique graph maximizing $P_G(k)$. We also consider generalizing $\ell$-connected graphs to connected graphs with minimum degree $\delta$.
    
\end{abstract}

\section{Introduction}

Let $G=(V,E)$ be a simple, undirected graph. A {\it $k$-coloring} of $G$ is an assignment $c:V\rightarrow [k]$ of colors to the vertices of $G$ such that no adjacent vertices are assigned the same color. Let $\chi(G)$ denote the chromatic number of $G$, which is the smallest $k$ for which such a coloring exists, and write $P_G(k)$ for the number of $k$-colorings of $G$. Note that $P_G(k)$ known to be a polynomial and is called the {\it chromatic polynomial} of $G$. We say a graph $G$ is {\it $k$-chromatic} if $\chi(G)=k$.

A number of recent papers have considered the question of maximizing the value of $P_{G}(k)$, where $G$ ranges over some specified family of graphs.  Examples here include results for the families of $n$-vertex $m$-edge graphs \cite{LPS,MN}, $n$-vertex regular graphs \cite{GalvinTetali,SSSZhao}, $n$-vertex graphs with fixed minimum degree \cite{Engbers,GuggiariScott}, and $n$-vertex $2$-(edge-)connected graphs \cite{EngGal,To94}. 

Most of the remainder of this note will focus on the family of $n$-vertex $k$-chromatic graphs.  Tomescu \cite{ToBook} showed that the disjoint union of a $k$-clique and $n-k$ isolated vertices maximizes $P_{G}(k)$ in this family. When restricting to the family of {\em connected} $n$-vertex $k$-chromatic graphs, the question of maximizing $P_{G}(k)$ seemed to be much more difficult, with an extremal value conjectured in 1971 by Tomescu \cite{To}.  Recently Fox, He, and Manners \cite{FoHeMa} resolved this conjecture, computing the maximum number of $k$-colorings of a connected graph with $n$ vertices and chromatic number $k\ge 4$. Recall that the $\ell$-core of a graph $G$ is the (unique) maximal subgraph with minimum degree at least $\ell$.

\begin{theorem}[\cite{FoHeMa}]\label{thm-FoHeMa}
If $G$ is a connected graph on $n$ vertices and $\chi(G)=k\ge 4$, then
\[
P_G(k) \le k!(k-1)^{n-k},
\]
with equality if and only if the $2$-core of $G$ is a $k$-clique.
\end{theorem}

In this paper, we investigate the question of maximizing $P_G(k)$ under stronger connectivity constraints. Recall that a graph $G$ is $\ell$-connected if it has more than $\ell$ vertices and it remains connected upon the removal of any set of fewer than $\ell$ vertices. We extend the methods of Fox, He, and Manners to compute the exact maximum value of $P_G(k)$ over $2$-connected $k$-chromatic graphs on a fixed number of vertices.

Recall that an ear of a graph is a path where the two endpoints may coincide. Let $G_{n,k}$ be the unique graph of order $n$ obtained from a $k$-clique by adding an ear (with $n-k$ internal vertices, each of degree $2$) attached to two distinct vertices of the clique.

\begin{theorem}\label{thm:2-conn}
If $G$ is a $2$-connected graph on $n$ vertices and $\chi(G)=k\geq 4$, then
\[
P_G(k)\leq (k-1)!((k-1)^{n-k+1}+(-1)^{n-k}),
\]
with equality if and only if $G\cong G_{n,k}$.
\end{theorem}

We also prove an asymptotically tight upper bound for $\ell$-connected graphs for all $\ell$ and provide examples of extremal graphs achieving the upper bound. 

\begin{theorem}\label{thm:l-conn}
Let $k\ge 4$, $\ell \ge 3$, and $n$ be sufficiently large in terms of $k$ and $\ell$. If $G$ is a $k$-chromatic $\ell$-connected graph on $n$ vertices, then
\[P_G(k)\leq k!(k-1)^{n-\ell -k +1}+O((k-2)^n).\]
Moreover, there exists a $k$-chromatic $\ell$-connected graph $G^*$ on $n$ vertices with $k\geq 4$ and $\ell \geq 3$ satisfying  $P_{G^*}(k)\geq k!(k-1)^{n- \ell -k +1}$.
\end{theorem}

In fact, our methods give an explicit structural description of all such $G^*$, see Theorem~\ref{thm:l-conn-stronger}, where we find the unique maximizing graph for some of the ranges of the parameters $k$ and $\ell$.

We also consider the related problem of bounding $P_G(k)$ when $G$ is connected and has minimum degree at least $\delta$. 

\begin{theorem}\label{thm-mindeg1}
Suppose $k-1\ge \delta \ge 3$ and $n$ is sufficiently large in terms of $\delta$ and $k$. Let $G^\star$ be obtained from a disjoint union of $K_{k}$ and $K_{\delta,n-k-\delta}$ by deleting an edge $e$ from $K_{\delta,n-k-\delta}$ and adding an edge from a vertex of $K_{k}$ to the endpoint of $e$ in the part of size $n-k-\delta$. Then $G^\star$ is the unique $n$-vertex $k$-chromatic minimum degree $\delta$ connected graph (up to isomorphism) with the maximum number of $k$-colorings. 
\end{theorem}

We will also be able to give a structural description of the extremal graphs in the other range $\delta \geq k \geq 4$, and to compute the first, second, and third order terms of the maximum value of $P_G(k)$ in the course of the proofs.

The paper is organized as follows. The next section collects the main lemmas of Fox, He, and Manners \cite{FoHeMa} which we require, and then sketches how to combine them to produce upper bounds on the chromatic polynomial $P_G(k)$. Section~\ref{sec:2-conn} adapts this machinery to prove Theorem~\ref{thm:2-conn}, the analogue of Tomescu's conjecture for $2$-connected graphs. Then, in Section~\ref{sec:min-deg}, we study the general case of graphs with minimum degree $\delta\ge 3$. We show that there are basically four different types of possible candidate graphs with the maximum number of colorings and we prove asymptotic bounds for such graphs (Lemma~\ref{lem:second-order}) which delivers Theorem \ref{thm-mindeg1}. Also, it turns out that only two types of these candidates can be $\delta$-connected. Using this, we study the problem for $\ell$-connected graphs in Section~\ref{sec:ell-con} and prove Theorem~\ref{thm:l-conn}. Finally, we mention some related open problems and conjectures in the final section.

\section{Background}\label{sec:background}

In this section, we state two results from Fox, He, and Manners \cite{FoHeMa}, which will be used in Section~\ref{sec:2-conn} to prove Theorem~\ref{thm:2-conn}. Then, we sketch the proofs of the main theorems in Sections~\ref{sec:2-conn} and~\ref{sec:min-deg}.

Recall that a $k$-chromatic graph is called {\it $k$-edge critical}, or simply {\it $k$-critical}, if removing any edge reduces its chromatic number.

Also, we say that a $k$-chromatic graph $G$ has property $C_k$ if, for every pair of distinct vertices $u,v\in V(G)$,
\[
\Pr_{c}[c(u)=c(v)] < \frac{1}{k-1},
\]
where the probability is taken over a uniform random $k$-coloring $c$ of $G$.

The first result we need shows that if a graph is both $k$-critical and satisfies property $C_k$, it cannot have many $k$-colorings.

\begin{theorem}[\cite{FoHeMa}]\label{fox_he_general}
Let $G$ be a $k$-critical graph of order $n$ where $n> k \geq 4$ and suppose $G$ satisfies property $C_k$. Then, $n\geq 2k-1$ and 
$$P_G(k)\leq  \begin{cases} 
      k!\left(\frac{7k+5}{12}+\frac{n(k-2)}{12(n-k)}\right)^{n-k} & \text{if}  \ 2k-1\leq n \leq k^2-k \\
       k!\left(\frac{k+1}{2}+\frac{n(k-2)}{6(n-k)}\right)^{n-k} & \text{if} \  n>k^2-k.
      \end{cases}$$
\end{theorem}

Theorem~\ref{fox_he_general} is proved exactly the same way as \cite[Lemma 8 and Theorem 10]{FoHeMa}; here we note that $G$ being $k$-critical is the essential property of a bad graph utilized in those proofs. The second result we need is a stronger bound that holds only for the case $k=4$.

\begin{theorem}[\cite{FoHeMa}]\label{fox_he_4crit}
Let $G$ be a $4$-critical graph of order $n\geq 6$ satisfying property $C_4$. Then,
\[
P_G(4)<4!\,2^{(11n-54)/12}\,3^{(2n-3)/6}.
\]
\end{theorem}

Theorem~\ref{fox_he_4crit} is proved the same way as \cite[Lemma 11]{FoHeMa}.

The proof of Theorem~\ref{thm:2-conn} relies on showing that a minimal counterexample graph $G$ must be both $k$-critical and satisfy property $C_k$. Roughly speaking, the graph $G$ will be $k$-critical because otherwise we can find a counterexample with fewer edges by removing an edge, and $G$ will satisfy property $C_k$ because otherwise we can find a counterexample with fewer vertices by identifying two vertices together. These arguments resemble those of \cite{FoHeMa} but are complicated by the fact that everything must respect the $2$-connectivity of $G$.

On the other hand, the proof of Theorem~\ref{thm:l-conn} is of a completely different flavor; in general, asymptotic results of this form are easier to prove than exact results like Theorem~\ref{thm:2-conn} since the hardest cases tend to occur when $n=\Theta(k)$. 

When $n$ is much larger than $k$, the structure of $G$ is much more tightly confined. In fact, we will prove that if a graph $G$ of order $n$ much larger than $k$ with minimum degree $\delta$ has many $k$-colorings, then $G$ contains a bipartite subgraph $K_{\delta, n-C}$ where $C>0$ depends only on $\delta$ and $k$. This will allow us to determine the approximate structure of every extremal graph $G$.
\section{$2$-connected Graphs}\label{sec:2-conn}
In this section, we prove Theorem~\ref{thm:2-conn}.

We say that $G$ is an {\it $(n,k)$-bad graph}  if it is a minimal counterexample for Theorem~\ref{thm:2-conn} with $n$ vertices and chromatic number $k$. 

In particular, suppose that Theorem \ref{thm:2-conn} is false for some $k \geq 4$, and fix this value of $k$.  Let $n>k$ be the minimum number of vertices in a counterexample $G$ to Theorem \ref{thm:2-conn} for this value of $k$, and we also assume that no proper subgraph of $G$ is a counterexample.  So an $(n,k)$-bad graph is an $n$-vertex $k$-chromatic $2$-connected graph $G$ so that $G$ either satisfies
\[
P_G(k) > (k-1)!\left((k-1)^{n-k+1}+(-1)^{n-k}\right),
\]
or satisfies
\[
P_G(k) = (k-1)!\left((k-1)^{n-k+1}+(-1)^{n-k}\right)
\]
and $G$ is not isomorphic to $G_{n,k}$, $n$ is the minimal such value for this fixed $k$, and every proper subgraph of $G$ satisfies Theorem \ref{thm:2-conn}. 

For the rest of this section, let 
\[
P_n(k) = P_{G_{n,k}}(k) = (k-1)!((k-1)^{n-k+1}+(-1)^{n-k}).
\]

Our first goal is to prove that every $(n,k)$-bad graph is $k$-critical. We start with an old lemma which bounds colorings of $2$-connected graphs.

\begin{lemma}[\cite{EngGal,To94}]\label{2connLem} If $G$ is a $2$-connected graph on $n$ vertices and $k$ is an integer with $k\geq 3$,  then 
\[
P_G(k)\leq (k-1)^{n}+(-1)^{n}(k-1).
\]
Moreover, for $n\neq 5$ or $k \neq 3$,  equality holds if and only if $G$ is the cycle $C_n$.
\end{lemma}

Next we prove a simple result about the equality case graph $G_{n,k}$. We use the standard notation $G + uv$ or $G+e$ for adding a single edge $e$ between nonadjacent vertices $u,v$ of $G$, and $G-uv$ or $G-e$ for the deletion of the edge $e$. We also write $G / uv$ for the graph obtained by contracting two vertices $u,v$ of $G$ together.

\begin{lemma}\label{kclique_ear_sub} Let $k \geq 3$, and let $G= G_{n,k}$. If $u$ and $v$ are two nonadjacent vertices of $G$, then $P_{G+uv}(k)<P_n(k).$
\end{lemma}

\begin{proof}
It is easy to see that $\chi(G/uv)=k$ since $k\geq 3$. So, $P_{G/uv}(k)>0$ and we have 
\[
P_{G+uv}(k)=P_G(k)-P_{G/uv}(k)<P_G(k)=P_n(k).
\]\end{proof}

We are now ready to show the first main lemma.

\begin{lemma}\label{criticalLemma} Let $n\geq k\geq 4$. If $G$ is an $(n,k)$-bad graph, then $G$ is $k$-critical. 
\end{lemma}

\begin{proof}
Suppose $G$ is not $k$-critical.  Let $e$ be an edge so that $\chi(G-e)=k$. First let us show that $G-e$ is not $2$-connected, so to that end suppose that $G-e$ is $2$-connected. Since no proper subgraph of $G$ is a counterexample for Theorem~\ref{thm:2-conn}, we have $P_{G-e}(k)\leq P_n(k)$. Now every proper $k$-coloring of $G$ is a proper $k$-coloring of $G-e$, so $P_G(k)\leq P_{G-e}(k)$. Since $G$ is a counterexample it follows that $P_G(k)= P_{G-e}(k)= P_n(k)$. As $G-e$ satisfies Theorem~\ref{thm:2-conn}, $G-e\cong G_{n,k}$. Hence, $G$ contains $G_{n,k}$ as a proper subgraph and so $P_G(k)<P_n(k)$ by Lemma~\ref{kclique_ear_sub}, which is a contradiction. Thus, $G-e$ is not $2$-connected.

Since $G-e$ is not $2$-connected and $G$ is $2$-connected, $G-e$ has exactly one cut-vertex $v$ and exactly two blocks, say $G_1$ and $G_2$. Let $n_1=|V(G_1)|$ and $n_2=|V(G_2)|$, and without loss of generality assume $G_1$ is $k$-chromatic. Note that we must have $n_1 \geq k$ and $n_2 \geq 2$.  By the minimality of $G$ we have $P_{G_1}(k)\leq P_{n_1}(k)$, and since $G_2$ is $2$-connected we have $P_{G_2}(k)\leq (k-1)^{n_2}+(-1)^{n_2}(k-1)$ by Lemma \ref{2connLem}. Therefore
$$P_{G-e}(k)=\frac{P_{G_1}(k)P_{G_2}(k)}{k}\leq \frac{P_{n_1}(k)\left((k-1)^{n_2}+(-1)^{n_2}(k-1)\right)}{k}.$$
Now we consider three cases.

\medskip

\textbf{Case 1:} Suppose that $n_2=2$.  In this case, the graph $G_2$ contains a vertex $w \neq v$ with the edge $e$ given by $ww_1$ for some $w_1$ in $G_1$. Recall that $P_{G_1}(k) \leq (k-1)!((k-1)^{n-k} + (-1)^{n-1-k})$, with equality if and only if $G\cong G_{n,k}$. 

\textbf{Subcase 1a:} Suppose $P_{G_1}(k)<(k-1)!((k-1)^{n-k}+(-1)^{n-1-k})$. Since $P_{G_1}(k)$ is divisible by $k!$ (by the symmetry of $k$-colorings), we have 
$$P_{G_1}(k) \leq (k-1)!((k-1)^{n-k}+(-1)^{n-1-k})-k! = (k-1)!((k-1)^{n-k}+(-1)^{n-1-k}-k).$$ 
This implies that 
\[
P_{G-e}(k) \leq (k-1)!((k-1)^{n-k} + (-1)^{n-1-k}-k)(k-1).
\]
But $P_n(k) = (k-1)!((k-1)^{n-k+1}+(-1)^{n-k})$, and therefore 
\[
P_{n}(k) - P_{G-e}(k) = (k-1)!((-1)^{n-k} - (-1)^{n-k-1}(k-1) + k(k-1))
\]
which is positive for $k \geq 4$.  This implies that $P_G(k) \leq P_{G-e}(k) < P_n(k)$, which contradicts that $G$ is $(n,k)$-bad.

\textbf{Subcase 1b:} Suppose $P_{G_1}(k) = (k-1)!((k-1)^{n-k}+(-1)^{n-1-k})$, and so $G_1\cong G_{n-1,k}$, and also assume that $w_1$ and $v$ are adjacent. This means that each coloring of $G_1$ gives $k-2$ choices for the color on $w$, so we directly compute that 
\[
P_{G}(k) = (k-1)!((k-1)^{n-k}+(-1)^{n-1-k})(k-2).
\]
Here $n \geq k+1$ (as $n_1 \geq k$ and $n_2 =2$) and $k \geq 4$ imply
\[
((k-1)^{n-k} + (-1)^{n-1-k})(k-2) = ((k-1)^{n-k+1}-(k-1)^{n-k}+(-1)^{n-1-k}(k-2)).
\]
If $n>k+1$, then $P_G(k)<P_n(k)$, which contradicts the assumption at $G$ is $(n,k)$-bad. If $n=k+1$, then $G\cong G_{n,k}$, and again $G$ cannot be $(n,k)$-bad.

\textbf{Subcase 1c:} Suppose $P_{G_1}(k) = (k-1)!((k-1)^{n-k}+(-1)^{n-1-k})$, and so $G_1\cong G_{n-1,k}$, and $w_1$ and $v$ are not adjacent.  This means that $w_1$ and $v$ lie on the ear of $G_1$, and $w$ is adjacent to both $w_1$ and $v$.

Now, we enumerate all $k$-colorings of $G$ by first coloring the $n-k+2$ vertices on the ear (including the endpoints in the clique, which require different colors) and $w$, and then using one of $(k-2)!$ colorings on the rest of the clique.  Since $k \geq 4$, Lemma \ref{2connLem} implies that a cycle $C_{n-k+2}$ has strictly more $k$-colorings than the $n-k+2$ vertices in the ear plus $w$.  Therefore $P_G(k)<P_n(k)$, which contradicts the assumption that $G$ is $(n,k)$-bad.

\medskip

\textbf{Case 2:} Suppose that $n_1>k$ and $n_2 \geq 3$. We claim that 
$$\frac{P_{n_1}(k)\left((k-1)^{n_2}+(-1)^{n_2}(k-1)\right)}{k}<P_n(k).$$
This is equivalent to
$$2(-1)^{n-k+1}+\frac{(-1)^{n-k+1}}{k-1}<(k-1)^{n-k}-(-1)^{n_2}(k-1)^{n_1-k+1}-(-1)^{n_1-k}(k-1)^{n_2-1},$$
so it suffices to check that 
$$2+\frac{1}{k-1}<(k-1)^{n_1-k}(k-1)^{n_2-1}-(k-1)^{n_1-k+1}-(k-1)^{n_2-1}$$
or

\[
k+1 + \frac{1}{k-1} < \left( (k-1)^{n_1-k} - 1 \right) \left( (k-1)^{n_2-1}-(k-1)\right).
\]
Since $n_2 \geq 3$ and $n_1-k>0$ (by assumption), we have $(k-1)^{n_2-1}-(k-1) \geq (k-1)(k-2)$ and $(k-1)^{n_1-k}-1 \geq k-2$. But the inequality 
\[
k+1+\frac{1}{k-1} < (k-2)(k-2)(k-1).
\]
holds for $k \geq 4$. 

Putting this together, we have $P_G(k)\leq P_{G-e}(k)<P_n(k)$ and this contradicts the assumption that $G$ is $(n,k)$-bad.
 
 \medskip

\textbf{Case 3:} Suppose that $n_1=k$ and $n_2 \geq 3$. In this case $G_1$ is a $k$-clique. Then $e=u_1u_2$ where $u_i\in V(G_i)\setminus \{v\}$. Let $G_2'$ be a subgraph of $G$ induced by the vertices of $G_2$ and $u_1$. Now $G$ is the union of a $k$-clique and a $2$-connected graph $G_2'$ and they overlap over the edge $u_1v$. Also, if $G_2'$ is a cycle then $G_2$ is a path. But $G_2$ is $2$-connected and hence an edge, which contradicts that $n_2 \geq 3$. So $G_2'$ cannot be a cycle, and thus by Lemma~\ref{2connLem}, $P_{G_2'}(k)<(k-1)^{n_2+1}+(-1)^{n_2+1}(k-1)$. Therefore,
$$P_{G}(k)=\frac{P_{G_1}(k)P_{G_2'}(k)}{k(k-1)}< \frac{k!((k-1)^{n_2+1}+(-1)^{n_2+1}(k-1))}{k(k-1)}.$$
It is clear that the latter is equal to $P_n(k)$.
Therefore we obtain $P_G(k)<P_n(k)$, which again contradicts the assumption that $G$ is $(n,k)$-bad. 
Thus $G$ must be $k$-critical.
\end{proof}

All that remains is to prove that $(n,k)$-bad graphs must satisfy the property $C_k$ which we defined in Section~\ref{sec:background}. We gather the results we need into several lemmas. Let $F$ be a set of edges  and $x,y$ be two distinct vertices  of $G$. We say that $F$ is a {\it disconnecting set} of edges if $G\setminus F$ is disconnected and $F$ is called an $x,y$ {\it disconnecting set of edges} if the vertices $x$ and $y$ belong to different components of $G\setminus F$. A graph $G$ is called {\it k-edge-connected} if every disconnecting set of edges has at least $k$ edges. It was proven by Dirac \cite{dirac} that every $k$-critical graph is $(k-1)$-edge-connected. Also, the edge version of the well known Menger's Theorem \cite{menger} says that the minimum size of an $x,y$ disconnecting set of edges is equal to the maximum number of pairwise edge disjoint paths joining $x$ to $y$. These two results together imply Lemma~\ref{CritProp}\ref{edgeDisPath} and their proofs also can be found  on pages $211$ and $168$ in \cite{westbook} respectively. Let $S$ be a set of vertices of a graph $G$. We say that $S$ is a {\it cut-set} of $G$ if $G\setminus S$ is disconnected.  An {\it $S$-lobe} of $G$ is an induced subgraph of $G$ whose vertex set consists of $S$ and the vertices of a connected component of $G\setminus S$. Lemma~\ref{CritProp}\ref{lobe} appears on page $218$ in \cite{westbook}.

\begin{lemma}[\cite{westbook}]\label{CritProp}
If $G$ is a $k$-critical graph, then
\begin{enumerate}[label={(\roman*)},itemindent=1em]
    \item every two distinct vertices of $G$ are joined by $k-1$ pairwise edge disjoint paths, and \label{edgeDisPath} 
    \item for every cut-set $S=\{x,y\}$ of $G$, $xy\notin E(G)$ and $G$ has exactly two $S$-lobes and they can be named $G_1$, $G_2$ such that $G_1+xy$ is $k$-critical and $G_2/xy$ is $k$-critical. \label{lobe}
 \end{enumerate}
\end{lemma}

A {\it theta graph} is obtained by joining end-vertices of three internally disjoint paths. We shall also use the following upper bound which follows immediately from Lemma~4.4 in \cite{ereyDM}. (Note that for a simple graph, at most one of the internally disjoint paths can be an edge.)

\begin{lemma}[\cite{ereyDM}]\label{theta}  If $G$ is a connected graph containing a theta subgraph, then 
$$P_G(k)\leq \frac{(k-1)^{|V(G)|+1}}{k}\left(1+\frac{3}{(k-1)^3}+\frac{1}{(k-1)^4}\right).$$
\end{lemma}

Lastly, we need the following result on the maximum number of colorings of a graph with chromatic number at least $3$.

\begin{lemma}[\cite{To90}]\label{3chrom} Let $G$ be a connected graph on $n$ vertices with $\chi(G)\geq 3$.
\begin{itemize}
\item If $n$  is odd, $P_G(k)\leq (k-1)^n-(k-1)$ with equality if and only if $G\cong C_n$.
\item If $n$ is even, $P_G(k)\leq (k-1)^n-(k-1)^2$ with equality if and only if $G$ is a cycle with a pendant vertex attached.
\end{itemize}
\end{lemma}

We are now ready to prove the next result that we will need.

\begin{lemma}\label{cont2conn} Let $n\geq k\geq 4$ and $G$ be an $(n,k)$-bad graph. Then $G/uv$ is $2$-connected for every pair of nonadjacent vertices $u$ and $v$.
\end{lemma}

\begin{proof}
By Lemma~\ref{criticalLemma} we may assume that $G$ is $k$-critical.
Suppose $G/uv$ is not $2$-connected for some pair of nonadjacent vertices $u$ and $v$, which means that $S=\{u,v\}$ is a cut-set of $G$. By Lemma~\ref{CritProp}\ref{lobe} the graph $G$ has two $S$-lobes $G_1$ and $G_2$ such that both $G_1+uv$ and   $G_2/uv$ are $k$-critical.  Let $n_1=|V(G_1)|$ and $n_2=|V(G_2)|$. Note that $\chi (G_1/uv), \, \chi(G_2+uv)\geq k-1$ and
\[
P_G(k)= P_{G+uv}(k)+P_{G/uv}(k)= \frac{P_{G_1+uv}(k)P_{G_2+uv}(k)}{k(k-1)}+\frac{P_{G_1/uv}(k)P_{G_2/uv}(k)}{k}.
\]

Now we consider several cases. In each case we will show that $P_G(k)<P_n(k)$  and this contradicts the fact that $G$ is $(n,k)$-bad.

\medskip

\textbf{Case 1}: $G_1+uv$ is a $k$-clique and $G_2/uv$ is not a $k$-clique.

Since $G_1+uv$ is a $k$-clique, it is clear that $G_1/uv$ is a $(k-1)$-clique, $P_{G_1+uv}(k)=k!$, and $P_{G_1/uv}(k)=(k-1)!$. Now let us show that $G_2+e$ has a theta subgraph. Since $G$ is $k$-critical, $N_G(u)\nsubseteq N_G(v)$ and $N_G(v)\nsubseteq N_G(u)$. Since $G_1+uv$ is a $k$-clique we have $N_{G_1}(u)= N_{G_1}(v)$, and so there are two vertices $u'$ and $v'$  such that $u'\in N_{G_2}(u)\setminus N_{G_2}(v)$ and $v'\in N_{G_2}(v)\setminus N_{G_2}(u)$. Let $w$ be the vertex in $G_2/uv$ obtained by contracting the vertices $u$ and $v$. Since $G_2/uv$ is $k$-critical and $k\geq 4$, by Lemma~\ref{CritProp}\ref{edgeDisPath} the vertices $u'$ and $v'$ are joined by three pairwise edge disjoint paths with one of them possibly being the path $u'wv'$. So, in $G_2+e$, the vertices $u$ and $v$ are  joined by three pairwise edge disjoint paths with one of them possibly being the path $u'uvv'$. Hence, $G_2+uv$ has a theta subgraph. So, by Lemma \ref{theta}, $P_{G_2+uv}(k)\leq \frac{(k-1)^{n_2+1}}{k}\left(1+\frac{3}{(k-1)^3}+\frac{1}{(k-1)^4}\right)$. Therefore,
$$P_{G+uv}(k)\leq\frac{(k-1)!}{k}\left((k-1)^{n_2}+3(k-1)^{n_2-3}+(k-1)^{n_2-4}\right).$$
Also, since $G_2/uv$ is not a counterexample for Theorem \ref{thm:2-conn} and the number of $k$-colorings of two $k$-chromatic graphs differs by a multiple of $k!$, it follows that $P_{G_2/uv}(k)\leq P_{n_2-1}(k)-k!$ (here we also have used the fact that $G_2/uv$ is $k$-critical and not a $k$-clique, so cannot contain a $k$-clique as a subgraph, and therefore $P_{G_2/uv}(k) \neq P_{n_2-1}(k)$). Therefore,
$$P_{G/uv}(k)\leq \frac{((k-1)!)^2}{k}((k-1)^{n_2-k}+(-1)^{n_2-1-k}-k).$$
Now we claim that the sum of the upper bounds obtained for $P_{G+uv}(k)$ and $P_{G/uv}(k)$ is less than $P_n(k)$. The latter is equivalent to 
$$(k-1)^{n_2-k}(3(k-1)^{k-3}+(k-1)^{k-4}+(k-1)!)-(k-1)!((-1)^{n_2-k}+k)<(k-1)^{n_2-1}+(-1)^{n_2}k.$$
It is trivial that $-(k-1)!((-1)^{n_2-k}+k)<(-1)^{n_2}k$ for $k \geq 4$, so it suffices to check that 
$$(k-1)^{n_2-k}(3(k-1)^{k-3}+(k-1)^{k-4}+(k-1)!)<(k-1)^{n_2-1}.$$
The latter is equivalent to
$$\frac{3}{(k-1)^2}+\frac{1}{(k-1)^3}+\frac{(k-1)!}{(k-1)^{k-1}}<1,$$
and for $k \geq 4$ we have $\frac{3}{(k-1)^2}\leq \frac{1}{3}$, $\frac{1}{(k-1)^3}\leq \frac{1}{27}$, and $\frac{(k-1)!}{(k-1)^{k-1}} \leq \frac{1}{3}$ and so the inequality holds, which implies $P_G(k)<P_n(k)$ in this case.

\medskip

\textbf{Case 2}: $G_2/uv$ is a $k$-clique and $G_1+uv$ is not a $k$-clique.

In this case we have $G_1+uv$ is not a counterexample for Theorem~\ref{thm:2-conn}, so $P_{G_1+uv}(k)\leq P_{n_1}(k)$. Since $G_1+uv$ is $k$-critical and not a $k$-clique, it cannot contain a $k$-clique as a proper subgraph, so as in the previous case we have
$P_{G_1+uv}(k)\leq P_{n_1}(k)-k!$. Since $G_2/uv$ is a $k$-clique and $k\geq 4$, either $u$ or $v$ has at least two neighbors in $V(G_2)\setminus \{u,v\}$. One can greedily color the graph $G_2+uv$ and  obtain at most $k!(k-1)(k-2)$ many $k$-colorings. This implies that that 
$$P_{G+uv}(k)\leq (P_{n_1}(k)-k!)(k-2)(k-1)!.$$
Since $G_2/uv$ is a $k$-clique, $P_{G_2/uv}(k)=k!$. Also, $\chi(G_1/uv)\geq 3$ and therefore  $P_{G_1/uv}(k)\leq (k-1)^{n_1-1}-(k-1)$ by Lemma~\ref{3chrom}. So,
$$P_{G/uv}(k)<(k-1)!\left((k-1)^{n_1-1}-(k-1)\right).$$
It is easy to check that $(k-1)!((k-2)(P_{n_1}(k)-k!)+(k-1)^{n_1-1}-(k-1))$ is strictly less than $P_n(k)$ for $k \geq 4$. Thus we get $P_G(k)<P_n(k)$ in this case.

\medskip

\textbf{Case 3}: Neither  $G_1+uv$ nor   $G_2/uv$ is a $k$-clique.

In this case we have $G_1+uv$  satisfies Theorem~\ref{thm:2-conn}, so $P_{G_1+uv}(k)\leq P_{n_1}(k)$. Since $G_1+uv$ is $k$-critical, it cannot contain a $k$-clique as a proper subgraph.  As in the previous cases, this implies that $P_{G_1+uv}(k)\leq P_{n_1}(k)-k!$. Similarly, we have $P_{G_2/uv}(k)\leq P_{n_2-1}(k)-k!$. Since $\chi (G_1/uv)\geq 3$ and $\chi(G_2+uv)\geq 3$, by Lemma~\ref{3chrom}, we have $P_{G_1/uv}(k)\leq(k-1)^{n_1-1}-(k-1)$ and $P_{G_2+uv}(k)\leq(k-1)^{n_2}-(k-1)$. So we get

$$P_{G+uv}(k)\leq\frac{(P_{n_1}(k)-k!)((k-1)^{n_2}-(k-1))}{k(k-1)}$$
and 
$$P_{G/uv}(k)\leq\frac{((k-1)^{n_1-1}-(k-1))(P_{n_2-1}(k)-k!)}{k}.$$

Now we claim that the sum of the upper bounds we obtained for $P_{G+uv}(k)$ and $P_{G/uv}(k)$ is less than $P_n(k)$.
The latter is equivalent to 
\[
k-(-1)^{n_1-k}+(k-1)(k-(-1)^{n_2-k-1})+(-1)^{n-k+1}(k-1)+(-1)^{n-k+1}
\]
being less than
\[
(k-1)^{n_1-k+1}+(k-(-1)^{n_1-k})(k-1)^{n_2-1}+(k-(-1)^{n_2-k-1})(k-1)^{n_1-1}+(k-1)^{n_2-k+1}
\]
and since $n_1 \geq k+1$ and $n_2 \geq k+2$ (as $G_1+uv$ is critical and not a clique, and $G_2/uv$ is critical and not a $k$-clique) it easy to check that this inequality holds. Therefore $P_{G}(k)<P_n(k)$ in this case.

\medskip

\textbf{Case 4}: Both  $G_1+uv$ and $G_2/uv$ are $k$-cliques.

From similar arguments given in the previous cases we have the bounds $P_{G+uv}(k)\leq k!(k-1)!(k-2)$ and $P_{G/uv}(k)=(k-1)!(k-1)!$.  Furthermore, we know that $n=2k+1$ in this case. It is easy to check that $k!(k-1)!(k-2)+(k-1)!(k-1)!$ is less than $(k-1)!\left((k-1)^{n-k+1}+(-1)^{n-k}\right)$, and so $P_{G}(k)<P_n(k)$ in this case.
\end{proof}

\begin{lemma}\label{2conn_contr_prob}
Let $n\geq k\geq 4$ and $G$ be an $(n,k)$-bad graph. Then $G$ satisfies property $C_k$.
\end{lemma}
\begin{proof}
Let $u$ and $v$ be nonadjacent vertices, and suppose on the contrary that $\Pr_{c}[c(u)=c(v)]\geq \frac{1}{k-1}$. Since $G$ is $(n,k)$-bad we have $P_G(k)\geq P_n(k)$.  Therefore 
\begin{eqnarray*}
P_{G/uv}(k) & =&  \Pr_c[c(u)=c(v)]\, P_G(k) \\
&\geq & \frac{1}{k-1} (k-1)!((k-1)^{n-k+1}+(-1)^{n-k}) \\
&=& (k-1)!\Big((k-1)^{n-k}+\frac{(-1)^{n-k}}{k-1}\Big).
\end{eqnarray*}

Now $G/uv$ is $k$-chromatic as $G$ is $k$-critical and $u$ and $v$ are not adjacent. Also, $G/uv$ is $2$-connected by Lemma~\ref{cont2conn}. If $n-k$ is even, it is clear that 
\[
(k-1)!\Big((k-1)^{n-k}+\frac{(-1)^{n-k}}{k-1}\Big) > (k-1)!((k-1)^{n-k}+(-1)^{n-k-1}).
\]

This means that $G/uv$ is a counterexample for Theorem~\ref{thm:2-conn} and this contradicts $G$ being a minimal counterexample. If $n-k$ is odd, $P_{G/uv}(k)\geq (k-1)!(k-1)^{n-k}-(k-2)!$ and the number of $k$-colorings of the extremal graph for Theorem~\ref{thm:2-conn} on $n-1$ vertices is $(k-1)!(k-1)^{n-k}+(k-1)!$. Now the number of $k$-colorings of two $k$-chromatic graphs differ by a multiple of $k!$, but the difference between $(k-1)!(k-1)^{n-k}-(k-2)!$ and $(k-1)!(k-1)^{n-k}+(k-1)!$ is $(k-1)!+(k-2)!$ which is strictly less than $k!$ as $k \geq 4$. This implies that in fact $P_{G/uv}(k)\geq (k-1)!(k-1)^{n-k}+(k-1)!$. 

Since $G/uv$ is not a counterexample,  $P_{G/uv}(k)= (k-1)!(k-1)^{n-k}+(k-1)!$ and also $G/uv\cong G_{n-1, k}$. Let $w$ be the vertex of $G/uv$ which is obtained by contracting $u$ and $v$. If $w$ does not belong to the $k$-clique of $G/uv$ then $G$ contains a $k$-clique and $G$ has more than $k$ vertices, which together contradict the fact that $G$ is $k$-critical.  So we now assume that $w$ belongs to the $k$-clique of $G/uv$. The remaining vertices of the clique belong to a $(k-1)$-clique of $G$.

In this case, we can build a $(k-1)$-coloring of $G$ as follows. Start by coloring the vertices of the $(k-1)$-clique. Next, color $u$ and $v$ the same color. The remaining vertices lie on the ear of $G/uv$, and can be colored accordingly since $k\ge 4$. Thus $G$ is not $k$-chromatic.
\end{proof}

We have all the necessary lemmas and theorems to prove the result for $2$-connected graphs.

\medskip

\noindent {\it Proof of Theorem~\ref{thm:2-conn}}. Suppose there exists an $(n,k)$-bad graph $G$. By Lemma~\ref{criticalLemma}, we know that $G$ is $k$-critical. Also, by Lemma~\ref{2conn_contr_prob}, we have $G$ satisfies property $C_k$. This means that Theorem \ref{fox_he_general} applies, and so in particular $n\geq 2k-1$ and the bounds given hold.  But it is straightforward to check that
$$ k!\left(\frac{k+1}{2}+\frac{n(k-2)}{6(n-k)}\right)^{n-k}<P_n(k)$$ holds for all $n>k^2-k$ when $k\geq 5$ and for $n\geq 18$ when $k=4$; and $$k!\left(\frac{7k+5}{12}+\frac{n(k-2)}{12(n-k)}\right)^{n-k}<P_n(k)$$ holds for $2k-1\leq n\leq k^2-k$ when $k\geq 6$ and for $n\geq 11$ when $k=5$. Thus, if there exists an $(n,k)$-bad graph then either $k=4$ and $7\leq n\leq 17$ or $k=5$ and $n=9,10$. Computer aided calculations show that there are no $(9,5)$, $(10,5)$ and $(n,4)$-bad graphs with $7\leq n\leq 10$. Also, by Theorem~\ref{fox_he_4crit}, if $G$ is an $(n,4)$-bad graph then $P_G(4)$ is less than $2^{(11n-54)/12}\,3^{(2n-3)/6}$ and it is straightforward to check  that the latter is less than $3!(3^{n-3}+(-1)^{n-4})$ for $11\leq n\leq 17$. Thus, there does not exist a counterexample for Theorem~\ref{thm:2-conn} and the result follows.

\section{Connected graphs with given minimum degree} \label{sec:min-deg}
In this section, we prove Theorems~\ref{thm:l-conn} and~\ref{thm-mindeg1} together. As an $\ell$-connected graph has minimum degree $\ell$, we will begin by proving structural results on the extremal graphs with a given minimum degree, then specialize to the $\ell$-connected setting.

We begin with the following definitions.
\begin{definition}
Fix $\delta\geq 3$ and $k\geq 4$. We say that a graph $G$ is an {\em $(n,k,\delta)$-graph} if it is a connected graph with $n$ vertices, minimum degree $\delta$, and chromatic number $k$.

Furthermore, $G$ is {\em $(n,k,\delta)$-maximum} if it is an $(n,k,\delta)$-graph and has the largest value of $P_G(k)$ among all $(n,k,\delta)$-graphs. 
\end{definition}
We will assume that $n$ is sufficiently large so that $(n,k,\delta)$-graphs exist. We now describe a ``typical" $(n,k,\delta)$-graph with many $k$-colorings. Let $G_1$ be the graph on a vertex set with three parts $X, Y, Z$ of orders $k-1$, $\delta$, and $n-\delta-k+1$, respectively, where the edges of $G_1$ are all pairs within $X$, so $G[X]$ is a clique, and all pairs from $Y$ to $X\cup Z$. This graph is easily seen to be an $(n,k,\delta)$-graph if $n$ is large.

We first remark that every proper $k$-coloring of $G_1$ uses $k-1$ colors on $X$, the set $Y$ is colored monochromatically with the last color, and any of the $k-1$ colors of $X$ can be used on $Z$.  This gives
\begin{equation}\label{eq-lowerbound}
P_{G_1}(k) = k! (k-1)^{n-\delta-k+1}.
\end{equation}
Our main result in this section is that (\ref{eq-lowerbound}) is asymptotically largest possible, and that any $(n,k,\delta)$-graph with at least this many colorings must have a structural decomposition similar to that of $G_1$. The optimal graphs $G^*$ in Theorem~\ref{thm:l-conn} and $G^\star$ in  Theorem~\ref{thm-mindeg1} will thus have a very similar three-part structure to $G_1$, and have slightly more $k$-colorings.

\begin{definition}
We say that a graph $G=(V,E)$ with minimum degree $\delta$ has an {\em $(X,Y,Z)$ decomposition} if $V=X\sqcup Y \sqcup Z$, $|Y| = \delta$, the induced subgraph $G[Y\cup Z]$ is the complete bipartite graph between $Y$ and $Z$, and there are no edges between $X$ and $Z$. See Figure \ref{fig-XYZdecomposition}.
\end{definition}

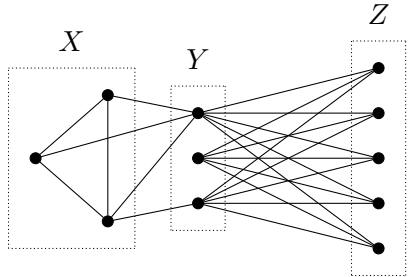
\begin{figure}[ht!]
\begin{center}
\begin{tikzpicture}[scale=1.2]

		\node (v1) at (1,.8) [circle,draw,scale=.4,fill] {};
		\node (v2) at (2,2) [circle,draw,scale=.4,fill] {};
		\node at (2,2.6) {$Y$};
		
		\draw[densely dotted] (1.7,2.3) -- (2.3,2.3);
		\draw[densely dotted] (2.3,2.3) -- (2.3,.7);
		\draw[densely dotted] (2.3,.7) -- (1.7,.7);
		\draw[densely dotted] (1.7,.7) -- (1.7,2.3);
		
		\node (v3) at (1,2.2) [circle,draw,scale=.4,fill] {};
		\draw[densely dotted] (-.1,2.5) -- (1.3,2.5);
		\draw[densely dotted] (1.3,2.5) -- (1.3,.5);
		\draw[densely dotted] (1.3,.5) -- (-.1,.5);
		\draw[densely dotted] (-.1,.5) -- (-.1,2.5);
		
		\node at (.6,2.8) {$X$};
		\node (v4) at (2,1) [circle,draw,scale=.4,fill] {};
		\node (v10) at (2,1.5) [circle,draw,scale=.4,fill] {};
		\node (v5) at (4,2) [circle,draw,scale=.4,fill] {};
		\node (v6) at (4,.5) [circle,draw,scale=.4,fill] {};
		\node (v7) at (4,1) [circle,draw,scale=.4,fill] {};
		\node (v8) at (4,2.5) [circle,draw,scale=.4,fill] {};
		\node (v0) at (4,1.5) [circle,draw,scale=.4,fill] {};
		\draw[densely dotted] (3.7,.2) -- (4.3,.2);
		\draw[densely dotted] (4.3,.2) -- (4.3,2.8);
		\draw[densely dotted] (4.3,2.8) -- (3.7,2.8);
		\draw[densely dotted] (3.7,2.8) -- (3.7,.2);
		
		\node at (4,3.1) {$Z$};
		\node (v9) at (0.2,1.5) [circle,draw,scale=.4,fill] {};

		\foreach \from/\to in {v9/v1,v9/v2,v9/v3,v1/v3,v1/v2,v2/v3,v1/v4,v4/v5,v4/v6,v4/v7,v2/v5,v2/v6,v2/v7,v2/v8,v4/v8,v0/v2,v0/v4,v10/v5,v10/v6,v10/v7,v10/v8,v10/v0}
		\draw (\from) -- (\to);
	\end{tikzpicture}
	\end{center}
	\caption{An example of an $(X,Y,Z)$ decomposition of a graph with minimum degree $\delta=3$.}
	\label{fig-XYZdecomposition}
\end{figure}

We will need two lemmas to prove such a decomposition exists. The first lemma will be used to bound the number of ways to color the internal vertices of a path whose endpoints have already been colored. It is stated in \cite{Engbers15} in terms of $H$-colorings, which for $H=K_{k}$ corresponds to a $k$-coloring.
\begin{lemma}[\cite{Engbers15}]\label{lem-paths}
Suppose that $r \geq 4$. The number of $k$-colorings of the path $P_r$ so that the endpoints are two fixed, predetermined colors is at most $((k-1)^2-1)(k-1)^{r-4}$.
\end{lemma}

We also use the following theorem of Erd\H{o}s and P\'osa. 
\begin{theorem}[\cite{ErdosPosa}]\label{thm-EP}
There is a function $f:\mathbb{N} \to \mathbb{R}$ such that a graph $G$ either contains $d$ disjoint cycles, or contains $f(d)$ vertices whose deletion makes the graph acyclic.
\end{theorem}

Given a set $C$ of colored vertices in a connected graph $G$, we {\em greedily color} the remaining vertices by considering them in order of non-decreasing distance from $C$, breaking ties arbitrarily. Each new uncolored vertex $v$ has a neighbor that has already been colored, and so has at most $k-1$ choices for its color.

\subsection{Structural Lemma for Extremal Graphs}

Now, we are ready to state and prove a structural lemma for all extremal graphs with at least as many $k$-colorings as $G_1$ from (\ref{eq-lowerbound}).

\begin{lemma}\label{lem-extremal}
Suppose $\delta \ge 3$, $k\ge 4$, and $n$ is sufficiently large in terms of $\delta$ and $k$. If $G$ is an $(n,k,\delta)$-graph and
\begin{equation}\label{eq-extremal}
P_G(k) \ge k! (k-1)^{n-\delta-k+1},
\end{equation}
then $G$ has an $(X,Y,Z)$ decomposition of one of the following four types:
\begin{enumerate}
    \item $G[X]\simeq K_{k-1}$ and there is a vertex of $Y$ complete to $X$.
    \item $G[X]\simeq K_{k-1}$ and for each $x\in X$, there is a vertex of $Y$ whose neighborhood in $X$ is $X\setminus \{x\}$.
    \item $G[X]\simeq K_k$ and there is exactly one vertex of $X$ with neighbors in $Y$.
    \item $G[X]$ is $K_k$ plus a leaf $v$, $v$ is the only vertex of $X$ with neighbors in $Y$, and $v$ has at least $\delta -1$ neighbors in $Y$.
\end{enumerate}
In addition, Case 2 is only possible when $k - 1 \le \delta$, while Cases $3$ and $4$ are only possible if $k-1 \ge \delta$.
\end{lemma}

\begin{proof}[Proof of Lemma~\ref{lem-extremal}]
We determine structures that must be present in $G$; the first several claims of this proof are inspired by the first several claims of the proof of Theorem 1.1 in \cite{GuggiariScott}. Throughout, we treat $\delta$ and $k$ as constants and all implicit constants are allowed to depend on them.

\bigskip

\noindent\textbf{Claim 1:} {\em $G$ has a bounded number of disjoint cycles.}

Suppose that $G$ has $c$ disjoint cycles.  We color $G$ by fixing a vertex $v$ in $G$ and coloring it arbitrarily.  Then we properly color a shortest path from $v$ (the current set of colored vertices) to a vertex on one of the disjoint cycles; each vertex on this path must avoid the color on the previous vertex in the path.  By Lemma \ref{lem-paths}, the vertices on the cycle can be colored in at most $((k-1)^2-1)(k-1)^{t-3}$ ways, where $t$ is the number of vertices in the cycle. Using the fact that $G$ is connected, we then repeat this process of finding a shortest path from the set of colored vertices to a disjoint cycle and coloring it in at most $((k-1)^2-1)(k-1)^{t-3}$ ways.  After all cycles have been colored, we greedily color the remaining vertices; each remaining vertex has at most $k-1$ choices for its color.  Since there are $c$ disjoint cycles, this process iterates $c$ times and so we have
\[
P_{G}(k) \leq k((k-1)^2-1)^c (k-1)^{n-2c-1} \leq k(k-1)^{n-1}e^{\frac{-c}{(k-1)^2}}.
\]
If
\[
c > (k-1)^2(k+\delta-2)\log(k-1), 
\]
then $P_{G}(k) < k(k-1)^{n-k-\delta+1}$, which contradicts (\ref{eq-extremal}) (here and throughout all logarithms are base $e$). Thus, there must be at most $(k-1)^2(k+\delta-2)\log(k-1)$ disjoint cycles in $G$.

\bigskip

We now use Theorem \ref{thm-EP} to identify a set $A$ of vertices of $G$, where by the previous claim $|A|$ is at most a constant depending of $k$ and $\delta$, so that the removal of the vertices of $A$ from $G$ makes the graph acyclic.  Therefore we can partition the vertices of $G$ into a set $A$ (where $|A|$ is a constant depending on $k$ and $\delta$) and a set $F$ so that $G[F]$ is a forest.  We will simply refer to $F$ as the forest.

Say a component of this forest $F$ is {\em non-trivial} if it contains some edge. Note that for every non-trivial component $T$ of $F$, all maximal paths of $T$ will have endpoints which are leaves of $T$, and all leaves of $T$ must have at least $\delta - 1 \ge 2$ neighbors in $A$. We use these facts in the next claim.

\bigskip

\noindent \textbf{Claim 2:} {\em The forest $F$ has a bounded number of non-trivial components.}

If there are $c$ non-trivial components, we consider a maximal path in each non-trivial component.  We first color $A$ in at most $k^{|A|}$ ways.  We then iteratively color the maximal paths, where the number of colorings for one maximal path is bounded above by using Lemma \ref{lem-paths}.  Finally, we color the remaining vertices greedily.  This gives
\[
P_{G}(k) \leq k^{|A|}((k-1)^2-1)^c(k-1)^{n-|A|-2c} \leq k^{|A|}(k-1)^{n-|A|}e^{\frac{-c}{(k-1)^2}}.
\]
If $c>|A|(k-1)^2\log(k/(k-1))+(k-1)^2(k+\delta-1)\log(k-1)$, then this implies that $P_G(k) < (k-1)^{n-\delta-k+1}$. Again, this contradicts (\ref{eq-extremal}), so we know that there are at most \[
|A|(k-1)^2\log(k/(k-1))+(k-1)^2(k+\delta-1)\log(k-1)
\]
non-trivial components in the forest $F$.

\bigskip

Now let $T$ be one fixed non-trivial component of $F$, so $T$ is a tree.  Let $T'$ be the subtree obtained from $T$ by deleting all of the leaves of $T$.  We will show that the number of vertices in $T'$ is bounded by a constant that depends on $k$ and $\delta$ by using the next two claims. Specifically, we show that the length of the longest path in $T$ is bounded, and that $T'$ has a bounded number of leaves, which together show that the number of vertices in $T'$ is bounded.

\bigskip

\noindent \textbf{Claim 3:} {\em The length of the longest path in $T$ is bounded.}

Suppose that the longest path $P$ in $T$ has at least $2c$ vertices.  Label the vertices of $P$ as $u_1 v_1 u_2 v_2 \cdots$. Again, first color the vertices in $A$.  Then we color $P$ as follows. Consider each pair $u_i$ and $v_i$ together for increasing values of $i$.  Given a pair $u_i$ and $v_i$, we see that $u_i$ has a neighbor that has already been colored ($u_1$ has a neighbor in $A$, and $u_i$ for $i>1$ has $v_{i-1}$ as a neighbor).  We then take a maximum length path $Q_i$ in the tree $T$ that starts at $v_i$ and otherwise avoids $P$; we use here that $\delta \geq 3$ implies $v_i$ has a neighbor that is not in $P$. The maximality of $Q_i$ implies that the other endpoint of this path must have a neighbor in $A$. Therefore the path $u_i \cup Q_i$ can be colored in at most $((k-1)^2-1)(k-1)^{|Q_i|-1}$ ways by Lemma \ref{lem-paths}. We then proceed to the next two vertices in $P$, and in this way we color the first $2c$ vertices of $P$. Notice that $T$ being a tree implies that the paths for $Q_i$ and $Q_j$ for $i \neq j$ do not intersect.

After greedily coloring the remaining vertices on the path $P$ and also of the rest of the graph, we have
\[
P_{G}(k) \leq k^{|A|}((k-1)^2-1)^c(k-1)^{n-|A|-2c} \leq k^{|A|}(k-1)^{n-|A|}e^{\frac{-c}{(k-1)^2}}.
\]
The same analysis given in Claim 2 shows that 
\[
c \leq |A|(k-1)^2\log(k/(k-1))+(k-1)^2(k+\delta-1)\log(k-1),
\]
and so the length of the longest path in $T$ is bounded above by a constant.

\bigskip

\noindent \textbf{Claim 4:} {\em $T'$ has a bounded number of vertices.}

We first show that $T'$ has a bounded number of leaves. Suppose that $T'$ has $c$ leaves. As $\delta \geq 3$, each leaf $v$ of $T'$ has at least two neighbors outside of $T'$, and by definition of $T'$ we have that at least one neighbor of $v$ is a leaf of $T$.  Note that a leaf of $T$ has a neighbor in $A$.  If two neighbors of a leaf $v$ of $T'$ are leaves of $T$, we have a path on 5 vertices that starts and ends in $A$ with $v$ in as the middle vertex. If only one neighbor of $v$ is a leaf of $T$, then another neighbor of $v$ is in $A$, so we have a path on 4 vertices that starts and ends in $A$. Here we used the fact that every leaf of $T$ has at least $\delta -1 \ge 2$ neighbors in $A$.

We first color $A$. For each leaf of $T'$, we color along the path in $T$ on 4 or 5 vertices containing it which has endpoints in $A$. This obtains an upper bound from Lemma \ref{lem-paths}. Finally, we greedily color the remaining vertices of the graph.  From this we have 
\[
P_{G}(k) \leq k^{|A|} ((k-1)^2-1)^c (k-1)^{n-2c-|A|} \leq k^{|A|} (k-1)^{n-|A|} e^{\frac{-c}{(k-1)^2}}.
\]
Again, the analysis of Claim 2 shows that 
\[
c \leq |A|(k-1)^2\log(k/(k-1))+(k-1)^2(k+\delta-1)\log(k-1), 
\]
and so the number of leaves in $T'$ is bounded above by a constant.

Finally, note that the maximum path of $T'$ is of bounded length, so $T'$ has bounded radius. Fixing an arbitrary root, we know that the number of vertices of any given distance from the root is at most the number of leaves; this implies $|V(T')|$ is bounded.

\bigskip

Now, for every component $T$, we move the set of nonleaves $V(T')$ into $A$; in the special case of $T' = \varnothing$ then $T$ is a single edge and we move one of the two vertices into $A$. Let $L$ be the resulting set containing $A$. This partitions the vertices of $G$ into a set $L$ and a set $R$ (which is a subset of the forest $F$) so that $R$ is an independent set. The size of $L$ is absolutely bounded in terms of $\delta$ and $k$ because $|A|$ is bounded,  the number of components is bounded, and the size of each $V(T')$ is bounded. Note that all neighbors of a vertex in $R$ must be in $L$. For each $Y\in {L \choose \delta}$ we define $Z_Y$ to be the set of all common neighbors of $Y$ in $R$. By the minimum degree condition, every vertex of $R$ has at least $\delta$ neighbors in $L$, so the set family $\{Z_Y: Y \in {L \choose \delta}\}$ covers $R$.

\bigskip

\noindent \textbf{Claim 5:} {\em For all but one choice of $Y$, $|Z_Y|$ is bounded.}

Suppose $Y,Y'\in{L\choose \delta}$ with $Y \neq Y'$. We break $k$-colorings of $G$ into three types: those that have at least two colors on $Y$, those that have at least two colors on $Y'$, and those that are monochromatic on $Y$ and $Y'$.

We count colorings of the first type by first coloring $Y$ arbitarily in at most $k^\delta$ ways, then coloring each vertex of $Z_Y$ arbitrarily in at most $k-2$ ways, and finally coloring the rest of $G$ greedily. The number of such colorings is at most
\[
k^\delta (k-2)^{|Z_Y|}(k-1)^{n-|Z_Y|-\delta} \le k^\delta (k-1)^{n-\delta} e^{-|Z_Y|/(k-1)}.
\]

By the same reasoning, the number of colorings of the second type is at most $k^\delta (k-1)^{n-\delta} e^{-|Z_{Y'}|/(k-1)}$.

If there's a coloring of the third type, then both $Y$ and $Y'$ are independent sets. The number of colorings of the third type is then the same as the number of colorings of the graph  obtained by contracting each of $Y$ and $Y'$ to a vertex (which will be the same vertex if $Y$ and $Y'$ intersect). Since this new graph is still connected, has chromatic number at least $k$, and has at most $n-\delta$ vertices because $Y \neq Y'$, we have by Theorem~\ref{thm-FoHeMa} that the number of colorings of the third type is at most $k!(k-1)^{n-\delta - k}$. In total, we get
\[
P_G(k) \le k^\delta (k-1)^{n-\delta} (e^{-|Z_Y|/(k-1)} + e^{-|Z_{Y'}|/(k-1)}) + k!(k-1)^{n-\delta - k},
\]
so if $|Z_Y|$ and $|Z_{Y'}|$ are sufficiently large in terms of $\delta$ and $k$, then $G$ does not satisfy (\ref{eq-extremal}). 

\bigskip

We have shown that there is a single set $Y\in {L\choose \delta}$ for which $|Z_{Y'}|$ is bounded in terms of $\delta$ and $k$ for every $Y'\neq Y$. Note that the argument for Claim 5 also shows that $Y$ must be an independent set so that there exists colorings monochromatic on $Y$. Furthermore, since the family $\{Z_Y' : Y' \in {L\choose \delta}\}$ covers all of $R$ and $|L|$ is itself bounded in terms of $\delta$ and $k$, $|Z_Y|\ge n - O(1)$. Define 
\[
Z=Z_Y \setminus (\bigcup_{Y'\ne Y} Z_{Y'}),
\]
and define $X=V(G)\setminus(Y\cup Z)$. We claim that this triple $(X,Y,Z)$ is an $(X,Y,Z)$ decomposition of $G$. Both $Y$ and $Z$ are independent sets, so it is certainly the case that $G[Y\cup Z]$ is a complete bipartite graph between $Y$ and $Z$. Also, every vertex in $X$ is either an element of $L\setminus Y$ or an element of $R\setminus Z_Y$, so there are no edges between $X$ and $Z$. Thus this is a valid $(X,Y,Z)$ decomposition of $G$ with $|Y|=\delta$ and $|X|$ bounded in terms of $\delta$ and $k$.

Next we will apply Theorem~\ref{thm-FoHeMa} to show that $G[X]$ is essentially a $k$-clique.

\bigskip

\noindent \textbf{Claim 6:} {\em The $2$-core of $G'=G[X\cup Y]/Y$ is a $k$-clique.}

We again break the $k$-colorings of $G$ into two types: those which have at least two colors on $Y$, and those which are monochromatic on $Y$. The number of colorings of the first type can be bounded by
\[
k^{|Y|}(k-1)^{|X|} (k-2)^{|Z|} = O((k-2)^{n}),
\]
which is negligibly small compared to $P_G(k)$ for $n$ sufficiently large.

The number of colorings of $G$ monochromatic on $Y$ is the same as the number of colorings of the graph $G/Y$ where $Y$ is contracted to a vertex. This graph has chromatic number at least $k$, and all the vertices of $Z$ have degree $1$ in it. Therefore, the number of colorings of $G/Y$ is just the number of colorings of $G' = G[X\cup Y]/Y$, which must also have chromatic number at least $k$, multiplied by $(k-1)^{|Z|}$. Since $G'$ is a connected graph on $|X| + 1$ vertices with chromatic number at least $k\ge 4$, we get by Theorem~\ref{thm-FoHeMa},
\[
P_{G'}(k)\le k!(k-1)^{|X| + 1 - k},
\]
with equality if and only if the $2$-core of $G'$ is a $k$-clique.

Suppose equality doesn't hold. Then, $P_{G'}(k)\le k!(k-1)^{|X| + 1 - k} -1$, so
\[
P_G(k) = P_{G/Y}(k) + O((k-2)^{n}) \le (k!(k-1)^{|X| + 1 - k} -1) \cdot (k-1)^{|Z|}+O((k-2)^n),
\]
and since $|Z| = n - O(1)$, this inequality contradicts (\ref{eq-extremal}) for $n$ sufficiently large. Therefore equality holds and this proves the claim.

\bigskip

It remains to show that if the $2$-core of $G'$ is a $k$-clique $C$, then $G[X\cup Y]$ must be one of the four types of graphs described. We may move all vertices of $X$ which only have neighbors in $Y$ to $Z$, guaranteeing that $G[X]$ is connected. Let $y$ be the vertex in $G'$ to which $Y$ is contracted. Then, because the original graph $G$ has minimum degree $\delta$, every vertex in $G'$ except $y$ and the neighbors of $y$ must have degree at least $\delta \ge 3$.

First, suppose $y$ is a vertex of $C$. We claim that $G'\simeq K_k$. If not, $G'$ is not equal to its own $2$-core, so it contains a vertex $v$ of degree $1$. By the previous observation, $v$ must be a neighbor of $y$, and $v$ must be complete to $Y$ in the original graph $G$, and have no other edges inside $G[X\cup Y]$. But any vertex whose neighborhood is exactly $Y$ is in $Z$, by the definition of $Z$, so such a $v$ does not exist.

Thus, $G'\simeq K_k$, which implies that $G[X]\simeq K_{k-1}$. In this scenario, we show that the Cases 1 and 2 are the only possible ways for $\chi(G)=k$ to hold. Indeed, suppose Cases 1 and 2 both fail to hold, so there exists a vertex $x\in X$ such that $X \setminus \{x\}$ has no common neighbor in $Y$. Then, give a $(k-1)$-coloring to $X$ where the color of $x$ is $k-1$. By the hypothesis, it is possible to extend this coloring to $Y$ while using only the first $k-2$ colors. Finally, coloring $Z$ with $k-1$ gives a proper $(k-1)$-coloring of $G$, which is the desired contradiction.

It remains to consider the case that $y$ is not a vertex of the $k$-clique $C$, so that $C$ is contained in $G'[X]$. We further split into cases based on the distance of $y$ from $C$. 

If $y$ is distance $1$ from $C$, then we claim that there are no other vertices in $G'$ besides $y$ and the clique. If not, some such vertex $v$ would have degree at most $1$. Thus, for $v$ to have degree at least $\delta$ in $G$, it must then be adjacent to $y$, and be complete to $Y$ in $G$. By the same argument as before, $v$ would then belong to $Z$ and not to $X\cup Y$ by the definition of $Z$. It follows that $G[X]\simeq K_k$ and exactly one vertex of $X$ has neighbors in $Y$. This falls into Case 3 of the lemma.

If $y$ is distance $2$ from $C$, then by a similar argument as the previous case there can be no other vertcies in $G'$ other than $C\cup \{y,v\}$, where $v$ is the unique vertex between $y$ and $C$. In this case $v$ must have at least $\delta-1$ edges to $Y$ to have degree at least $\delta$ in $G$, so $G$ falls into Case 4 of the lemma.

The final case is when $y$ is distance $3$ or more from $C$. We claim that this is impossible. Indeed, the graph $G''=G'\setminus \{y\}\cup N(y)$ also has $2$-core equal to $C$. Also, since every vertex of $\{y\}\cup N(y)$ has distance $2$ or more from $C$, it follows that $G''$ contains at least one additional vertex. In particular, $G''$ contains a vertex $v$ of degree $1$.

But since we deleted $\{y\}\cup N(y)$ from $G'$, this means that $v$ is not adjacent to $y$ in $G'$, so its degree in $G'$ is the same as its degree in $G$. Since this degree must be at least $\delta \ge 3$, we see that $v$ has at least two neighbors $w, w'$ in $N(y)$. But then $v, w, w', y$ form a four-cycle in $G'$ disjoint from $C$, which contradicts the fact that $C$ is the $2$-core of $G'$. Thus it is impossible for $y$ to be distance $3$ or more from $C$, and the four cases above exhaust all the possibilities for $G[X\cup Y]$.

In Case 2, there must be at least $k-1$ vertices in $Y$ and so $|Y|=\delta \geq k-1$. In Cases 3 and 4, there are always vertices of degree $k-1$, and so these two cases can only occur when $k-1 \ge \delta$.
\end{proof}

We make a simple observation about how to estimate the number of $k$-colorings with an $(X,Y,Z)$ decomposition. Note that given the sizes of the parts $X,Y,Z$ in the decomposition, the only remaining information needed to determine $G$ is the induced subgraph $G[X \cup Y]$.

\begin{lemma}\label{lem-p-series}
Suppose $G$ is a graph with an $(X,Y,Z)$ decomposition. Then,
\begin{equation}\label{eq-p-series}
P_G(k) = \sum_{i=1}^{k} P^{(i)} \cdot (k-i)^{|Z|},
\end{equation}
where $P^{(i)}$ is the number $k$-colorings of $G[X\cup Y]$ using exactly $i$ colors on $Y$.
\end{lemma}
\begin{proof}
Make each $k$-coloring of $G$ by first coloring $G[X\cup Y]$, then extending to $Z$ greedily. Since every vertex of $Z$ has neighborhood $Y$, it follows that the number of ways to extend a given coloring to $Z$ is exactly $(k-i)^{|Z|}$ if exactly $i$ colors already appear in $Y$.
\end{proof}

Due to Lemma~\ref{lem-extremal}, the extremal graphs we need to consider all have an $(X,Y,Z)$ decomposition where $|X|$ and $|Y|$ are bounded in terms of $\delta$ and $k$. Thus, the $i$-th term in the series (\ref{eq-p-series}) grows like $\Theta((k-i)^n)$ when $k,\delta$ are fixed and $n\rightarrow \infty$. 

\begin{notation}
When $G$ has an $(X,Y,Z)$ decomposition, we write $P_G^{(i)}(k)=P^{(i)}\cdot (k-i)^{|Z|}$ for the $i$-th order term, where $P^{(i)}$ is as defined in Lemma~\ref{lem-p-series}.
\end{notation}

Using Lemma~\ref{lem-extremal}, we now can not only show that (\ref{eq-extremal}) is asymptotically best possible, but also compute the optimal second-order term and estimate the optimal third-order term.

\subsection{The Second-Order Terms}
In the previous section, we determined the first-order terms of $P_G(k)$ for the four types of extremal graphs corresponding to the cases in Lemma~\ref{lem-extremal}.

\begin{definition}
A graph is of {\em Type $t$} if it satisfies the conditions of Case $t$ in Lemma~\ref{lem-extremal} for $t=1,2,3,4$, and we say that the graph is {\em maximum of Type $t$} if it has the most $k$-colorings amongst all $(n,k,\delta)$-graphs of Type $t$.
\end{definition}

We now compare the second-order terms $P_G^{(2)}(k)$ for these four types of graphs and show that when they exist, maximum graphs of Type 4 have the most $k$-colorings, and otherwise maximum graphs of Type 1 or 2 have the most.

\begin{lemma}\label{lem:second-order}
Suppose $\delta \ge 3$, $k\ge 4$, and $n$ is sufficiently large in terms of $\delta$ and $k$. If $G$ is an $(n,k,\delta)$-graph, then the second order term of $P_G(k)$ grows as follows.
\begin{enumerate}
    \item If $G$ is maximum of Type $1$, then
        \[ P_G^{(2)}(k) = k!(2^{\min(k-2, \delta - 1)} - 1)(k-1)(k-2)^{n-\delta -k + 1}.\]
    \item If $k-1 \le \delta$ and $G$ is maximum of Type $2$, then
        \[ P_G^{(2)}(k) = k!(2^{\min(k-2, \delta - k + 2)} - 1) (k-1)(k-2)^{n-\delta-k+1}.\]
    \item If $k-1 \ge \delta$ and $G$ is maximum of Type $3$, then
        \[ P_G^{(2)}(k) = k!(2^{\delta - 1} - 1)(k-1)^2(k-2)^{n-\delta - k}. \]
    \item If $k-1 \ge \delta$ and $G$ is maximum of Type $4$, then
        \[ P_G^{(2)}(k) = k!(2^{\delta -1} - 1)(k-1)^2(k-2)^{n-\delta-k} + k!(k-1)^2(k-2)^{n-\delta-k-1}. \]
\end{enumerate}
\end{lemma}
\begin{proof}
If $G$ is maximum of Type $t$, $t=1,2,3,4$, observe that the type of the graph determines all of its edges except those between $X$ and $Y$. Also, no proper spanning subgraph of $G$ can be a graph of Type $t$, since removing edges from $G$ always increases $P_G(k)$. Finally, every graph of each type is already guaranteed to have $n$ vertices, chromatic number $k$, and first order term $P_G^{(1)}(k) = k!(k-1)^{n-\delta - k + 1}$. These three observations together are enough to determine the graphs of each type that maximize the second-order term.

By Lemma~\ref{lem-p-series}, it suffices to maximize the number of $k$-colorings of $G[X\cup Y]$ which have exactly two colors in $Y$, subject to the constraint that the minimum degree is at least $\delta$.

\bigskip

\noindent \textbf{Type $1$.} {\em $G[X]\simeq K_{k-1}$ and there is a vertex of $Y$ complete to $X$.}

Here $|Z| = n-\delta-k+1$.

Let $y_1$ be a vertex of $Y$ complete to $X$.

Suppose first $k - 1 \ge \delta$. Then, the Type $1$ graph $G_1$ whose only edges between $X$ and $Y$ are the edges incident to $y_1$ already has minimum degree $\delta$. Since this is a spanning subgraph of every other graph of Type $1$, it is the Type $1$ graph maximizing $P_G(k)$.

The number of $k$-colorings of $G[X\cup Y]$ with exactly two colors on $Y$ can be computed by arbitrarily coloring $X$ in $k!$ ways, and then picking the two colors to appear on $Y$. The color of $y_1$ is determined to be the unique color not appearing in $X$. There are $k-1$ choices of the other color to appear in $Y$, and $2^{\delta-1} - 1$ ways to color the $\delta - 1$ remaining vertices of $Y$ so that both colors appear. Thus,

\[
P_G^{(2)}(k) = k!(k-1)(2^{\delta - 1} - 1) \cdot (k-2)^{|Z|} = k!(2^{\delta - 1} - 1)(k-1)(k-2)^{n-\delta - k + 1},
\]
as desired.

If $k-1 < \delta$, then we can remove edges from between $X$ and $Y$ until every vertex in $X$ has degree $\delta$. Notice that since vertices of $Y$ have degree at least $|Z| = n-k-\delta+1\ge \delta$ as $n$ is sufficiently large, so the entire graph still has minimum degree $\delta$. Let $S_1,\ldots, S_{k-1}\subseteq Y$ be the neighborhoods within $Y$ of the $k-1$ vertices $x_1,\ldots ,x_{k-1}$ in $X$. To $k$-color $G[X\cup Y]$, we again have $k!$ ways to pick the colors on $X$ and $k-1$ ways to pick the second color that appears in $Y$. If it is the color of $x_i$, then the colors of the vertices in $S_i$ are determined. There remains $2^{\delta - |S_i|} - 1$ ways to color the vertices of $Y\setminus S_i$. Since $|S_i| = \delta - k + 2$, we have

\[
P_G^{(2)}(k) = k!(2^{k-2} - 1)(k-1)(k-2)^{n-\delta - k + 1}
\]
in this case.
\bigskip

\noindent \textbf{Type $2$.} {\em $G[X]\simeq K_{k-1}$ and for each $x\in X$, there is a vertex of $Y$ whose neighborhood in $X$ is $X\setminus \{x\}$.}

Here $|Z| = n-\delta-k+1$.

Let the vertices of $X$ be $x_1,\ldots, x_{k-1}$ and let their neighborhoods within $Y$ be $S_1,\ldots, S_{k-1}$. For each $i$, $|S_i|\ge k-2$ by the definition of Type $2$ graphs. By a similar argument as for Type $1$ graphs, we can delete edges until $|S_i| = \max(k-2, \delta - k+2)$. 

Color $X$ in $k!$ ways. Since $\delta \geq 3$, the color not used on $X$ must appear in $Y$ and we have $k-1$ ways to choose the other color for $Y$. Once this other color is chosen, there are $2^{\delta - |Y_i|}$ ways to color the rest of $Y$. Thus,
\[
P_G^{(2)}(k) = k!(k-1)(2^{\delta - \max(k-2, \delta-k+2)} - 1)\cdot(k-2)^{|Z|} = k!(2^{\min(k-2, \delta - k + 2)} - 1) \cdot (k-1)(k-2)^{n-\delta-k+1}
\]
as desired.
\bigskip

\noindent \textbf{Type $3$.} {\em $G[X]\simeq K_k$ and there is exactly one vertex of $X$ with neighbors in $Y$.}

Here $|Z| = n-\delta-k$.

In this case the maximum will be achieved when there is exactly one edge $(x_1, y_1)$ between $X$ and $Y$. Color $X$ in $k!$ ways, and pick a color for $y_1$ in one of $k-1$ ways. There are $k-1$ choices for the other color to appear on $Y$, and for each such choice $2^{\delta -1} - 1$ ways to color the rest of $Y$. Thus,
\[
P_G^{(2)}(k) = k!(2^{\delta - 1} - 1)(k-1)^2(k-2)^{n-\delta - k}.
\]

\bigskip

\noindent \textbf{Type $4$.} {\em $G[X]$ is $K_k$ plus a leaf $v$, $v$ is the only vertex of $X$ with neighbors in $Y$, and $v$ has at least $\delta -1$ neighbors in $Y$.}

Here $|Z| = n-\delta-k-1$.

The maximum is attained when $v$ has exactly $\delta -1$ neighbors in $Y$. Let $y_1$ be the unique vertex in $Y$ not adjacent to $v$. Again, we count colorings of $G[X\cup Y]$ for which two colors appear on $Y$. Color $X$ in $k!$ ways and $v$ in $k-1$ ways, and let $c(v)$ be the color used on $v$.

First, we count the number of ways to color $Y$ in two colors where $c(v)$ is not used. In this case, we can pick any two colors in ${k-1 \choose 2}$ ways, and color $Y$ in a total of $2^\delta - 2$ ways.

On the other hand, if $c(v)$ is used, then it can only be used on $y_1$. We pick the other color that appears out of $k-1$ remaining colors, and color the rest of $Y$ all in that color.

In total,
\begin{align*}
P_G(k) & = k!(k-1)\Big[{k-1 \choose 2} (2^\delta - 2) + (k-1)\Big](k-2)^{n-\delta-k-1} \\
& = k!(2^{\delta -1} - 1)(k-1)^2(k-2)^{n-\delta-k} + k!(k-1)^2(k-2)^{n-\delta-k-1}.
\end{align*}
\end{proof}

The proof of Theorem \ref{thm-mindeg1} is now immediate. 

\begin{proof}[Proof of Theorem \ref{thm-mindeg1}]
This is a direct corollary of Lemma \ref{lem:second-order}. The maximizing graph of Type 4 is isomorphic to the graph $G^\star$ described.
\end{proof}

\subsection{The Third-Order Term}
It remains to study the maximum graphs of Types $1$ and $2$ in the case $k-1 < \delta$.

We will need the following combinatorial optimization problem. Recall that $S_i \triangle S_j$ is the symmetric difference of two sets $S_i$ and $S_j$.

\begin{definition}
If $r\ge 1, s\ge t \ge 1$, and $S_1,\ldots, S_r\in {[s] \choose t}$, define
\[
c(S_1,\ldots ,S_r) = \sum_{1\le i < j \le r} 3^{s - |S_i \cup S_j|} \cdot 2^{|S_i \triangle S_j|}.
\]
Let $c(r, s, t)$ be the maximum value of $c(S_1,\ldots ,S_r)$ over all possible choices of $S_1,\ldots, S_r \in {[s] \choose t}$.
\end{definition}

Suppose $G$ is of Type $1$ or $2$. Then $G[X]\simeq K_{k-1}$ and $G$ is uniquely determined by the $k-1$ sets $N(x_i)\cap Y$, where $x_i\in X$. We identify $Y$ with the set of positive integers $[\delta]$, and define the \textit{signature} of $G$ to be the sequence of $k-1$ subsets $N(x_i)\cap Y$ for $i \in [k-1]$. 

\begin{lemma}\label{lem:signature}
Suppose $\delta \ge k\ge 4$, $n$ is sufficiently large in terms of $\delta$ and $k$, $G$ is an $(n,k,\delta)$-graph that is maximum of Type $1$ or $2$, and $S_1,\ldots, S_{k-1}$ is the signature of $G$ considered as a sequence of subsets of $[\delta]$. Unless $k=4$ and $G$ is of Type $2$,
\[
P_G^{(3)}(k) = \Big[c(S_1,\ldots, S_{k-1}) - (k-1)(k-2)\cdot 2^{k-2} + {k-1 \choose 2}\Big](k-3)^{n-k-\delta +1}.
\]
\end{lemma}
\begin{proof}
By Lemma~\ref{lem-p-series}, $P_G^{(3)}(k) = P^{(3)}(k-3)^{n-k-\delta+1}$, where $P^{(3)}$ is the number of $k$-colorings of $G[X\cup Y]$ using exactly $3$ colors on $Y$. Since $G$ is a maximum graph of Type $1$ or $2$, we may assume that the degree of every vertex in $X$ is exactly $\delta$, and so $|S_i| = \delta - k +2$ for all $i$.

We condition on the three colors that appear on $Y$. 

Let the colors of $G[X]\simeq K_{k-1}$ be $1,\ldots, k-1$ and let $k$ be the distinguished color missing from $X$. We claim that if at most $3$ colors appear on $Y$, then color $k$ must appear on $Y$. Indeed, if $G$ is of Type $1$, then there is a vertex of $Y$ complete to $X$ so that vertex must have color $k$. On the other hand, if $G$ is of Type $2$, then for each vertex $x_i$ of $X$ there is a vertex $y_i$ of $Y$ whose neighborhood in $X$ is $X-\{x_i\}$. If color $k$ does not appear on $Y$, then the only choice for the color of $y_i$ is the color of $x_i$, so all $k-1$ other colors appear on $Y$. Since we assumed that if $G$ is of Type $2$ then $k\ne 4$, this means that at least $k-1\ge 4$ colors appear on $Y$, contradiction. Thus color $k$ is one of the three colors that must appear on $Y$.

Suppose the three colors that appear on $Y$ are $i$, $j$, and our distinguished color $k$. We now compute $P^{(3)}$ using inclusion-exclusion. Let $P^{C}$ be the number of ways to color $G[X\cup Y]$ using only the colors in a set $C\subseteq[k]$ to color $Y$. Then,
\begin{equation}\label{eq-PIE}
P^{(3)} = \sum_{1\le i < j \le k-1} (P^{\{i,j,k\}} - P^{\{i,k\}} - P^{\{j,k\}} + P^{\{k\}})
\end{equation}
by the principle of inclusion-exclusion. Note that every term contains $k$ because $k$ must appear on $Y$. It is easy to see that $P^{\{k\}} = 1$, and that $P^{\{i,k\}} = 2^{\delta - |S_i|} = 2^{k-2}$ for all $i$. Finally, note that there are $3$ available colors for each vertex outside $S_i \cup S_j$, $2$ colors for vertices in exactly one of them, and one color for vertices in both, so
\[
P^{\{i, j, k\}} = 3^{\delta - |S_i\cup S_j|}\cdot 2^{|S_i \triangle S_j|}.
\]
Plugging into (\ref{eq-PIE}) and summing over all $i, j$, this completes the proof.
\end{proof}

It follows that the optimal graphs of either type must maximize $c(S_1,\ldots, S_{k-1})$. Of course, they are subject to additional constraints: for graphs of Type $1$, the intersection of all the $S_i$ is nonempty, and for graphs of Type $2$, there is a distinct element $y_i$ in each of the $(k-2)$-fold intersections of the $S_i$. Taking these constraints into account, we can reduce the problem into solving for $c(r,s,t)$.

\begin{lemma}\label{lem-mindegthirdorder}
 Suppose $2k-4 > \delta \ge k\ge 4$ or $2k-4=\delta$ and $k \geq 5$, and $n$ is sufficiently large in terms of $\delta$ and $k$. If $G$ is an $(n,k,\delta)$-maximum graph, then $G$ is of Type $1$ and
\[
P_G^{(3)}(k) = \Big[c(k-1, \delta - 1, \delta - k + 1) - (k-1)(k-2)\cdot 2^{k-2} + {k-1 \choose 2}\Big](k-3)^{n-k-\delta +1}.
\]
\end{lemma}
\begin{proof}
By Lemma~\ref{lem:signature}, if $G$ is a $(n,k, \delta)$-maximum graph of Type $1$ then its signature $S_1,\ldots, S_{k-1}$ is a sequence of $(\delta-k+2)$-subsets of $Y\simeq [\delta]$ that maximizes the value of $c(S_1,\ldots ,S_{k-1})$, where by definition of Type 1 graphs we have $\{y\} \subseteq \bigcap_i S_i$. Removing $y$ from all the sets, we see that
\[
c(S_1,\ldots ,S_{k-1}) = c(S_1 - \{y\}, S_2 - \{y\},\ldots, S_{k-1} - \{y\}) \le c(k-1, \delta-1, \delta - k + 1),
\]
with equality if and only if $S_1 - \{y\}, S_2 - \{y\},\ldots, S_{k-1} - \{y\}$ is a sequence of sets maximizing $c(S_1 - \{y\}, S_2 - \{y\},\ldots, S_{k-1} - \{y\})$. Thus, the lemma is true if the $(n,k,\delta)$-maximum graph is of Type $1$.

It remains to show that the maximum graphs of Type $2$ have fewer colorings than maximum graphs of Type $1$ under the stated hypotheses. By Lemma~\ref{lem:second-order}, the second order term of $P_G(k)$ is smaller for maximum graphs of Type $2$ than maximum graphs of Type $1$ if $\delta < 2k-4$.

For $\delta=2k-4$ with $k \geq 5$, it is easy to compute that for a maximum graph of Type $2$ we have 
\[
c(S_1,\ldots ,S_{k-1}) = {k-1 \choose 2} \cdot 4 \cdot 3^{\delta - k + 1}.
\]
If we consider the specific Type 1 graph that joins $x_i$ both to vertex $y_1$ and to vertices $y_{i+1}, y_{i+2},\ldots,y_{i+k-3}$, then for all $i<j$ we have $|S_i \cup S_j|\geq k-1$ and also $|S_1 \cup S_3| >k-1$, and so it follows that
\[
c(S_1,\ldots,S_{k-1}) > {k-1 \choose 2} \cdot 4 \cdot 3^{\delta - k + 1}.
\]
As this is one specific Type 1 graph, the result follows in this case.
\end{proof}

We note that when $\delta=k=4$, a short computation gives $P^{(3)}(4) = 15$ for maximum graphs of Type 1 and $P^{(3)}(4) = 18$ for maximum graphs of Type 2, and so Type 2 graphs win in this case.

\subsection{$\ell$-connected graphs with $\ell \geq 3$}\label{sec:ell-con}

In this section we prove our main result for $\ell$-connected graphs with $\ell \geq 3$ (Theorem~\ref{thm:l-conn}). In fact, we show the following.

\begin{theorem}\label{thm:l-conn-stronger}
Let $k\ge 4$, $\ell \ge 3$, and $n$ be sufficiently large in terms of $k$ and $\ell$. Let $G$ be a $k$-chromatic $\ell$-connected graph on $n$ vertices.
\begin{itemize}
    \item [a)] Then,
\[P_G(k)\leq k!(k-1)^{n-\ell -k +1}+O((k-2)^n).\]
Moreover, if $G^*$ is a $k$-chromatic $\ell$-connected graph on $n$ vertices with the most $k$-colorings, $k\geq 4$ and $\ell \geq 3$, then $P_{G^*}(k)\geq k!(k-1)^{n- \ell -k +1}$.
    \item [b)] If $k\geq \ell$  then
\[P_G(k)\leq k!\left((k-1)^{n-\ell-k+1}+(2^{\ell -2}(2k-\ell -1)-k+1)(k-2)^{n-\ell-k+1}\right)+O((k-3)^n).\]
Moreover, the extremal graph $G^*$ is unique and is the Type 1 graph where $Y\setminus \{y_1\}$ is matched with $\ell-1$ vertices in $X$.
\item[c)] If $k< \ell$  then
\[P_G(k)\leq k!\left((k-1)^{n-\ell-k+1}+(k-1)(2^{k-2}-1)(k-2)^{n-\ell-k+1}\right)+O((k-3)^n).\]
Moreover, for $k< \ell < 2k-4$, equality holds only for some Type 1 graphs, and for $\ell > 2k-4$, there exist both Type 1 and Type 2 graphs satisfying the equality.
\item[d)] If $\ell \geq (k-2)(k-1)+1$, then
\[P_G(k)\leq k!(P^{(1)}+P^{(2)}+P^{(3)})+O((k-4)^n)\]
where $P^{(1)}=(k-1)^{n-\ell-k+1}$, $P^{(2)}=(k-1)(2^{k-2}-1)(k-2)^{n-\ell-k+1}$ and $P^{(3)}={k-1\choose 2}(2^{2k-4}-2^{k-1}+1)$.
Moreover, equality is achieved if and only if $G$ is the Type 1 graph where $N_Y(x_i)=Y\setminus \{y_{(i-1)(k-2)+1},\dots , y_{i(k-2)}\}$ for every $i=1,\dots ,k-1$.
\end{itemize}

\end{theorem}

First, we shall show that the first order term of the number of $k$-colorings of an asymptotically extremal graph must be equal to $k!(k-1)^{n-\ell -k+1}$.

\begin{proof}[Proof of Theorem~\ref{thm:l-conn-stronger}(a)]
Suppose that $G$ is an $n$-vertex $k$-chromatic $\ell$-connected graph.  Then $G$ is an $(n,k,\ell)$-graph, and so by Lemma~\ref{lem-extremal} we know that the maximum number of colorings occurs for a graph that is of one of the four Types described.  Note that Types $3$ and $4$ are not $\ell$-connected, as deleting a particular vertex disconnects the graph.  So  only graphs $G$ with $(X,Y,Z)$ decomposition of Type 1 or Type 2 may achieve this bound. The first order terms of the chromatic polynomials of both types of graphs are equal to $k!(k-1)^{n-\ell-k+1}$, as this is precisely the number of $k$-colorings of $G$ using exactly one color in $Y$.
\end{proof}

Now, to find the graphs with most $k$-colorings for large enough $n$, we shall compare the second order terms of the Type 1 and Type 2 graphs, or third order terms when the second order terms coincide. Every $k$-coloring of $G$ can be extended from a $k$-coloring of $G[X]$ as $G[X]$ is a clique and  there are $k!$ ways to $k$-color vertices in $X$.  To $k$-color the vertices in $Y$ using exactly two colors, we need a distinguished color which is not used for any of the vertices of $X$ as $G$ is $k$-chromatic and a color of a particular vertex $x$ of $X$. Let $N_Y(x)$ denote $N_G(x)\cap Y$. Since $N_Y(x)\neq \emptyset$, vertices of every subset of $Y\setminus N_Y(x)$ except the empty set  can be assigned the color of $x$ and the remaining vertices are assigned the distinguished color. So there are $\sum_{x \in X} \left( 2^{\ell-|N_Y(x)|} - 1 \right)$ ways to $k$-color the vertices of $Y$ using exactly two colors. Lastly, there are $(k-2)^{n-\ell-k+1}$ ways to color the vertices of $Z$ as $G[X\cup Y]$ is a complete bipartite graph. Therefore, 
\begin{equation}\label{2nd_ord_form}
  P_G^{(2)}(k) = k! \left(\sum_{x \in X} 2^{\ell-|N_Y(x)|} - 1 \right)(k-2)^{n-\ell-k+1}.  
\end{equation}

Also, let $X=\{x_1 ,\dots , x_{k-1}\}$ and $Y=\{y_1 ,\dots , y_{\ell}\}$. By the proof of Lemma~\ref{lem:signature}, the third order term is
\[P_G^{(3)}(k)=k! P^{(3)} (k-3)^{n-\ell -k+1}.\] where
\begin{equation}\label{3rd_ord_form}
  P^{(3)}=  \sum _{1\leq i < j\leq k-1}\left(3^{\ell - |N_Y(x_i)\cup N_Y(x_j)|}\,2^{|N_Y(x_i)\Delta N_Y(x_j)|}-2^{\ell-|N_Y(x_i)|}-2^{\ell-|N_Y(x_j)|}+1\right).
\end{equation}

Lastly, an $(n,k,\ell)$-graph of Type 1 or Type 2 is $\ell$-connected if and only if
\begin{itemize}
    \item  $(k-1)-|S| + |N_Y(S)| \geq \ell$ for each $S \subseteq X$ with $S \neq \varnothing$ and $S \neq X$, and
    \item  $|T| \leq |N_{X}(T)|$ for each $T \subseteq Y$ with $N_{X}(T) \neq X$.
\end{itemize}

The first condition holds, as deleting $X\setminus S$ and $N_Y(S)$ disconnects the graph. The second one holds, as deleting $Y\setminus T$ and $N_{X}(T)$ disconnects the graph and so $(\ell -|T|)+|N_X(T)|\geq \ell$.

\begin{proof}[Proof of Theorem~\ref{thm:l-conn-stronger}~(b, c, d)]
b) First suppose that $k\geq \ell +2$. There are no Type 2 graphs for this range so we only consider Type 1 graphs. By (\ref{2nd_ord_form}), we shall find the maximum value of $\sum_{x\in X}2^{-{|N_y(x)|}}$ to maximize the second order term. Let  $G^*$ be the Type 1 graph described in the theorem. It is $\ell$-connected  and $\sum_{x\in X}2^{-{|N_Y(x)|}}=\frac{1}{2}(k-\ell)+\frac{1}{4}(\ell -1)$.  If $G$ is an $\ell$-connected graph of Type 1 with a subset $T\subseteq Y\setminus \{y_1\}$ such that $N_X(T)=X$ then  every vertex in $X$ has at least two neighbors in $Y$ so such graph cannot achieve the maximum value of $\sum_{x\in X}2^{-{|N_Y(x)|}}$. Now suppose that $G$ is an $\ell$-connected graph of Type 1 with no subset  $T\subseteq Y\setminus \{y_1\}$ such that $N_X(T)=X$. We have $|N_X(T)|\geq |T|$ for every $T\subseteq Y\setminus \{y_1\}$, since $G$ is $\ell$-connected. By Hall's marriage theorem, there exists a matching between $X$ and $Y\setminus \{y_1\}$ with an edge incident to each vertex of $Y\setminus\{y_1\}$. Therefore $G$ contains $G^*$ as a spanning subgraph.

It remains to check that $G=G^*$ exactly, which would follow if every strict supergraph of $G^*$ obtained by adding edges between $X$ and $Y$ has strictly fewer $k$-colorings. Equivalently, for any nonadjacent pair $(x,y)\in X\cup Y$, we need to exhibit a $k$-coloring $\chi$ of $G^*$ such that $\chi(x)=\chi(y)$. To do so, color $X$ with the first $k-1$ colors, color $y$ with $\chi(x)$, and color the rest of $Y$ with the last color $k$. Since only two colors appear on $Y$, there is at least one color left for $Z$, so this gives a valid $k$-coloring. 

Thus, every strict supergraph of $G^*$ has strictly fewer $k$-colorings than $G^*$, so $G$ cannot achieve the maximum value unless it is isomorphic to $G^*$.

For $k=\ell +1$ or $k= \ell$, Type 2 graphs exist. However, for Type 2 graphs, $|N_Y(x)|\geq k-2\geq 2$ for each $x$ in $X$. If $k=\ell +1$ then the Type $1$ graph $G^*$ has one vertex $x\in X$ with exactly one neighbor in $Y$ and so the graph $G^*$ achieves the maximum value again. If $k=\ell$ then Type 2 graphs have at least one vertex $x\in X$ with three neighbors in $Y$ and again the graph $G^*$ achieves the maximum value.

c) In this case both Type 1 and Type 2 graphs are possible; we need to make sure that they are $\ell$-connected. For $\ell < 2k-4$, we have that Type 2 graphs have $|N_Y(x)| \geq k-2$. The Type 1 graph that joins $x_i$ both to vertex $y_1$ and to vertices $y_{i+1}, y_{i+2},\ldots,y_{\ell-k+i+1}$ is $\ell$-connected and has $|N_Y(x_i)| = \ell-k+2$.  As $\ell<2k-4$, we have Type 1 graphs maximize $P_G^{(2)}(k)$ in this range. 

For $\ell \geq 2k-4$, there exist $\ell$-connected graphs of both Type 1 and Type 2 satisfying $|N_Y(x)|=\ell -k+2$ for all $x\in X$. 

d) For a $k$-chromatic $\ell$-connected graph of Type 1 or 2, we have $|N_Y(x)|\geq \ell -k+2$ for each $x\in X$, as they have minimum degree $\ell$. Let $G^*$ be the Type 1 graph described. Note that $|N_Y(x)|=\ell -k+2$ for each $x\in X$ for $G^*$, so it maximizes the second order term. Also, when $|N_Y(x)|$ is fixed, maximizing the third order term reduces to  maximizing $\sum_{1\leq i< j\leq k-1}\left( \frac{3}{4}\right)^{|N_Y(x_i)\cap N_Y(x_j)|}$ by the formula in (\ref{3rd_ord_form}). Observe that $|N_Y(x_i)\cap N_Y(x_j)|\geq \ell -2k+4$ for every graph of Type 1 or 2 with minimum degree $\ell$. Moreover, if a Type 2 graph $G^{**}$ satisfies $|N_Y(x_i)\cap N_Y(x_j)|= \ell -2k+4$ for every pair  $x_i$, $x_j$ in $X$, then $G^{**}$ must be isomorphic to $G^*$ since $\ell \geq (k-1)(k-2)+1$. So, $G^*$ is the unique (up to isomorphism) graph achieving the maximum value of the third order term. This third order term is given by $k!\, \left({k-1\choose 2}(2^{2k-4}-2^{k-1}+1)\right)$ via a calculation from (\ref{3rd_ord_form}). 

\end{proof}

In the case of $4$-chromatic graphs, we find the graphs with maximum number of $4$-colorings for every $\ell$. 

\begin{corollary}\label{cor-k=4}
Let $G$ be a $4$-chromatic $\ell$-connected graph on $n$ vertices with $n$ sufficiently large.
\begin{itemize}
    \item [a)] If $3\leq \ell \leq 4$, then
        \[P_G(4)\leq 4!\,(3^{n-\ell -3}+5\cdot 2^{n-\ell -3}+P^{(3)})\]
        where $P^{(3)}=3(3^{\ell -1}-2^{\ell}+1)-{\ell -1 \choose 2}(5\cdot 3^{\ell -3}-2^{\ell -1})+(\ell -1)(\ell-4)(3^{\ell -2}-2^{\ell -2}).$
    Moreover, equality holds if and only if $G$ is the Type 1 graph where every vertex in $Y\setminus \{y_1\}$ has exactly one neighbor in $X$.
    \item[b)] If $5\leq \ell \leq 6$, then
    \[P_G(4)\leq 4!\,(3^{n-\ell -3}+5\cdot 2^{n-\ell -3}+P^{(3)})\]
        where $P^{(3)}=35$ if $\ell =5$ and $P^{(3)}=27$ if $\ell =6$. Moreover, for $\ell =5$ equality holds if and only if $G$ is the Type 2 graph where $N_Y(x_i)=Y\setminus \{y_i, y_{6-i}\}$ for $i=1,2$ and $N_Y(x_3)=Y\setminus \{y_3, y_5\}$, and for $\ell =6$ equality holds if and only if $G$ is the Type 2 graph where $N_Y(x_i)=Y\setminus \{y_i, y_{7-i}\}$ for $i=1,2,3$.
     \item [c)] If $\ell \geq 7$, then
    \[P_G(4)\leq 4!\,(3^{n-\ell -3}+5\cdot 2^{n-\ell -3}+27) \]
    with equality if and only if  $G$ is the Type $1$ graph described in Theorem~\ref{thm:l-conn-stronger}(d).
    \end{itemize}
\end{corollary}
\begin{proof}
a) The result follows immediately from Theorem~\ref{thm:l-conn-stronger}(b). The third order term of the extremal graph can be calculated using the formula in (\ref{3rd_ord_form}).

b) For $\ell =5$, the Type 2 graph described has minimum degree $3$, so it maximizes the second order term. Also,   we have $|N_Y(x_1)\cap N_Y(x_2)|=|N_Y(x_2)\cap N_Y(x_3)|=1$ for the Type $2$ graph. Every $5$-connected Type 1 graph contains at least two pairs $x_i,x_j$ with $|N_Y(x_i)\cap N_Y(x_j)|=2$. Therefore the Type 2 graph uniquely maximizes the third order term. For $\ell =6$, the Type $2$ graph has $|N_Y(x_i)\cap N_Y(x_j)|=2$ for every pair. However for every $6$-connected Type $1$ graph there exist a pair with $|N_Y(x_i)\cap N_Y(x_j)|=3$ so Type $2$ graph uniquely maximizes the third order term again.

c) The result follows immediately from Theorem~\ref{thm:l-conn-stronger}(d), as $(k-1)(k-2)+1=7$ for $k=4$.
\end{proof}

\section{Closing Remarks}

We end with several conjectures that are related to the contents of this paper. The following conjecture of Tomescu is still open.

\begin{conjecture} \label{conj:tomescu-general}
If $x\ge k\ge 4$ and $G$ is a connected graph on $n$ vertices with $\chi(G)=k$, then
\[
P_G(k) \le (x)_k(x-1)^{n-k}
\]
with equality if and only if the $2$-core of $G$ is a $k$-clique.
\end{conjecture}
The cases $k=4$ and $k=5$ have been verified by Knox and Mohar \cite{KnMo1, KnMo2}.  The generalization of Theorem \ref{thm:2-conn} to general $x$-colorings is also open.

\begin{conjecture}[\cite{BrownErey}]\label{conj:tomescu-2conn}
If $x\ge k\ge 4$ are integers and $G$ is a $2$-connected graph on $n$ vertices and $\chi(G)=k$, then
\[
P_G(x)\leq (x)_k((x-1)^{n-k+1}+(-1)^{n-k})
\]
with equality if and only if $G\cong G_{n,k}$.
\end{conjecture}

We make the following generalization of Conjectures \ref{conj:tomescu-general} and \ref{conj:tomescu-2conn} to the $\ell$-connected case.  Theorem \ref{thm:l-conn-stronger}(a) proves the case $x=k$.

\begin{conjecture}
Let $G$ be a $k$-chromatic $\ell$-connected graph on $n$ vertices with $k\geq 4$ and $\ell \geq 3$.
 Then, 
\[P_G(x)\leq (x)_k(x-1)^{n-\ell -k +1}+O((x-2)^n).\]
for every integer $x\geq k$.
\end{conjecture}

The proof of Lemma~\ref{lem-extremal} in this paper relies on Theorem \ref{thm-FoHeMa} (specifically when showing the 2-core of $G[X\cup Y]/Y$ is a $k$-clique); the rest of the proofs can be shown to hold for $x \geq k$ where the implied constants now also depend on $x$. So an extension of Theorem \ref{thm-FoHeMa} to general $x>k$ will extend Lemma~\ref{lem-extremal} to $x>k$, where $n$ will be sufficiently large depending on fixed $\delta$, $k$, and $x$.

We also ask for the graphs that maximize the number of proper $k$-colorings when $\ell >k$. In Theorem~\ref{thm:l-conn-stronger}, we found the unique extremal graph for $k\geq \ell$ and $\ell\geq (k-1)(k-2)+1$, and determined the approximate structure of the extremal graphs for $k<\ell < 2k-4$.  We showed that when $k<\ell <2k-4$ an extremal graph is of Type 1, however we did not determine its precise structure and if it is unique. Also, we leave the problem of determining the maximizing graph to be of Type 1 or Type 2, and which specific graph achieves the maximum value, unsolved for $2k-4\leq \ell \leq  (k-1)(k-2)$.


\begin{thebibliography}{10}

\bibitem{BrownErey} J. Brown, A. Erey, New bounds for chromatic polynomials and chromatic roots, {\it Discrete Math.} {\bf 338} (2015), 1938\textendash 1946.

\bibitem{dirac} G. A. Dirac, The structure of $k$-chromatic graphs, {\it Fund. Math.} {\bf 40} (1953), 42--55.

\bibitem{Engbers15}
J. Engbers, Extremal $H$-colourings of graphs with fixed minimum degree, {\em J. Graph Theory} {\bf 79} (2015), 103--124.

\bibitem{Engbers}
J. Engbers, Maximizing $H$-colorings of connected graphs with fixed minimum degree, {\em J. Graph
Theory} {\bf 85} (2017), 780--787.

\bibitem{EngbersErey} J. Engbers and A. Erey, Extremal colorings and independent sets, {\it Graphs Combin.}, {\bf 34} (2018), 1347--1361.

\bibitem{EngGal}
J. Engbers and D. Galvin, Extremal $H$-colorings of trees and 2-connected graphs, {\em J. Comb.
Theory Ser. B} {\bf 122} (2017) 800--814.

\bibitem{ErdosPosa}
P. Erd\H{o}s and L. P\'{o}sa, On independent circuits contained in a graph, {\em Canad. J. Math.} {\bf 17} (1965), 347--352.

\bibitem{ereyDM} A. Erey, Maximizing the number of $x$-colorings of $4$-chromatic graphs, {\it Discrete Math.} {\bf 341} (2018), 1419\textendash 1431.

\bibitem{ereyJOC} A. Erey, On the maximum number of colorings of a graph, {\it J. Combin.} {\bf 9} (2018) 489\textendash 497.

\bibitem{FoHeMa} J. Fox, X. He, and F. Manners, A proof of Tomescu's graph coloring conjecture, {\it J. Comb. Theory Ser. B} {\bf 136} (2019) 204\textendash 221.

\bibitem{GalvinTetali}
D. Galvin and P. Tetali, On weighted graph homomorphisms, DIMACS Series in Discrete
Mathematics and Theoretical Computer Science {\bf 63} (2004) {\em Graphs, Morphisms and Statistical
Physics}, 97--104.

\bibitem{GuggiariScott} H. Guggiari and A. Scott, Maximising $H$-colourings of graphs, {\em J. Graph Theory} {\bf 92} (2019), 172--185.

\bibitem{KnMo1}F. Knox and B. Mohar, Maximum number of colourings,
I. 4-chromatic graphs, preprint arXiv:1708.01781.

\bibitem{KnMo2}F. Knox and B. Mohar, Maximum number of colourings,
II. 5-chromatic graphs, {\em Electronic J. Combin.} {\bf 26(3)} (2019), \#P3.40.

\bibitem{LPS}
P.-S. Loh, O. Pikhurko, and B. Sudakov, Maximizing the number of $q$-colorings, {\em Proc. Lon.
Math. Soc.} {\bf 101} (2010), 655--696.

\bibitem{MN}
J. Ma and H. Naves, Maximizing proper colorings on graphs, {\em J. Combin. Theory Ser. B} {\bf 115}
(2015), 236--275.

\bibitem{menger}
K. Menger, Zur allgemeinen Kurventheorie, {\em Fund. Math.} {\bf 10}
(1927), 95--115.

\bibitem{SSSZhao}
A. Sah, M. Sawhney, D. Stoner, and Y. Zhao, A reverse Sidorenko inequality, preprint arXiv:1809.09462.

\bibitem{ToBook}
I. Tomescu, {\bf Introduction to Combinatorics}, Collets (Publishers) Ltd., London and Wellingborough,
1975.

\bibitem{To}I. Tomescu, Le nombre des graphes connexes $k$-chromatiques
minimaux aux sommets \'etiquet\'es, \emph{C. R. Acad. Sci. Paris} \textbf{273}
(1971), 1124\textendash 1126. 

\bibitem{To90}I. Tomescu, Maximal chromatic polynomials of connected planar graphs, {\it J. Graph Theory} {\bf 14} (1990), 101\textendash 110.

\bibitem{To94}I. Tomescu, Maximum chromatic polynomials of 2-connected graphs, {\it J. Graph Theory} {\bf 18} (1994), 329\textendash 336.

\bibitem{westbook}
D. B. West, {\bf Introduction to Graph Theory}, second ed., Prentice Hall, New York, 2001.

\end{thebibliography}
\end{document}